\documentclass[fleqn]{zzz}
\usepackage{mathrsfs}
\usepackage{xcolor}
\usepackage{tikz}
\usepackage{amsmath,amsthm,amsbsy,amssymb}
\usepackage{graphicx}
\usepackage{rotating}
\usepackage{fixmath}
\usepackage[pdftex]{hyperref}
\usepackage{soul}
\usepackage{stmaryrd}
\usepackage{relsize}
\usepackage{amssymb}
\usepackage{bm}
\usepackage{comment}
\usepackage{soul}
\usetikzlibrary{decorations}

\pgfdeclaredecoration{dashsoliddouble}{initial}{
 \state{initial}[width=\pgfdecoratedinputsegmentlength]{
 \pgfmathsetlengthmacro\lw{.3pt+.5\pgflinewidth}
 \begin{pgfscope}
 \pgfpathmoveto{\pgfpoint{0pt}{\lw}}%
 \pgfpathlineto{\pgfpoint{\pgfdecoratedinputsegmentlength}{\lw}}%
 \pgfmathtruncatemacro\dashnum{%
 round((\pgfdecoratedinputsegmentlength-3pt)/6pt)
 }
 \pgfmathsetmacro\dashscale{%
 \pgfdecoratedinputsegmentlength/(\dashnum*6pt + 3pt)
 }
 \pgfmathsetlengthmacro\dashunit{3pt*\dashscale}
 \pgfsetdash{{\dashunit}{\dashunit}}{0pt}
 \pgfusepath{stroke}
 \pgfsetdash{}{0pt}
 \pgfpathmoveto{\pgfpoint{0pt}{-\lw}}%
 \pgfpathlineto{\pgfpoint{\pgfdecoratedinputsegmentlength}{-\lw}}% 
 \pgfusepath{stroke}
 \end{pgfscope}
 }
}

\hypersetup{
colorlinks = true,
urlcolor = blue,
linkcolor = black,
}

\definecolor{Mycolor2}{HTML}{e85d04}
\sethlcolor{black}
 
\newcommand{\qpsihyp}[5]{\,\mbox{}_{#1}\psi_{#2}\!\left(\!\!\begin{array}{c}{#3}\\[0.10cm]{#4}\end{array};{#5}\right)}

\newcommand{\hyp}[5]{\,\mbox{}_{#1}F_{#2}\!\left(\!\!
\begin{array}{c}{#3}\\ {#4}\end{array};#5\right)}
\newcommand{\Hhyp}[5]{\,\mbox{}_{#1}H_{#2}\!\left(\!
\begin{array}{c}{#3}\\[0.08cm] {#4}\end{array};#5\right)}

\newtheorem{thm}{Theorem}[section]

\newtheorem{lem}[thm]{Lemma}

\makeatletter
\def\eqnarray{\stepcounter{equation}\let\@currentlabel=\theequation
\global\@eqnswtrue
\tabskip\@centering\let\\=\@eqncr
$$\halign to\displaywidth\bgroup\hfil\global\@eqcnt\z@
$\displaystyle\tabskip\z@{##}$&\global\@eqcnt\@ne
\hfil$\displaystyle{{}##{}}$\hfil
&\global\@eqcnt\tw@ $\displaystyle{##}$\hfil
\tabskip\@centering&\llap{##}\tabskip\z@\cr}

\def\endeqnarray{\@@eqncr\egroup
\global\advance\c@equation\m@ne$$\global\@ignoretrue}

\def\@yeqncr{\@ifnextchar [{\@xeqncr}{\@xeqncr[5pt]}}
\makeatother

\newcommand{\Z}{\mathbb{Z}} %INTEGERS
\newcommand{\R}{\mathbb{R}} %REAL NUMBERS
\newcommand{\C}{\mathbb{C}} %COMPLEX NUMBERS
\newcommand{\N}{\mathbb{N}} %POSITIVE NATURAL NUMBERS, n = 1, 2, ...
\newcommand{\CC}{{{\mathbb C}}}

\newcommand{\expe}{{\mathrm e}}

\newcommand{\dd}{{\mathrm d}}

\DeclareMathOperator{\Li}{Li_2}

\usepackage{scalerel,url}
\let\svus_
\catcode`_=\active %
\def_{\ifmmode\expandafter\svus\expandafter\bgroup\expandafter\lowerit
\else\svus\fi}
\def\lowerit#1{\ThisStyle{\raisebox{-2\LMpt}{$\SavedStyle#1$}}\egroup}
\makeatletter
\AtBeginDocument{
\check@mathfonts
\fontdimen13\textfont2=3.5pt
\fontdimen14\textfont2=3.5pt
\fontdimen16\textfont2=2.5pt
\fontdimen17\textfont2=2.5pt
 }
\makeatother

\begin{document}
\renewcommand{\PaperNumber}{***}

\FirstPageHeading

\ShortArticleName{Evaluation of beta integrals of Ramanujan type}

%%%%%%%%%%%%%%%%%%%%%%%%%%%%%%%%%%%%%%%%%%%%%%%%%%%%%%%%%%%%%
\ArticleName{Evaluation of beta integrals of Ramanujan type and integral representations for bilateral hypergeometric series}
% Names of the authors for the title of the paper
\Author{Howard S. Cohl$\,^{\ast}\orcidB{}$
and Hans Volkmer$\,^{\dag}\orcidD{}$
%Please carefully check the accuracy of 
%names and affiliations. 
%$\,^{\dag}$ 
%Roberto S. Costas-Santos$\,^{\dag}\orcidA{}$
%and Xiang-Sheng Wang $^{\ddag}\orcidC{}$
}

\AuthorNameForHeading{H.~S.~Cohl, H.~Volkmer}

\Address{$^\ast$ Applied and Computational 
Mathematics Division, National Institute 
of Standards 
and Tech\-no\-lo\-gy, Mission Viejo, CA 92694, USA
%Address of First Author, Country
\URLaddressD{
\href{http://www.nist.gov/itl/math/msg/howard-s-cohl.cfm}
{http://www.nist.gov/itl/math/msg/howard-s-cohl.cfm}
}
} % Address of First Author
\EmailD{howard.cohl@nist.gov} % E-mail address of First Author

\Address{$^\dag$ Department of Mathematical Sciences, University of Wisconsin–Milwaukee, PO Box 413,
Milwaukee, WI 53201, USA
%Address of First Author, Country
\URLaddressD{
\href{https://uwm.edu/math/people/volkmer-hans/}
{https://uwm.edu/math/people/volkmer-hans/}
}
} % Address of First Author
\EmailD{volkmer@uwm.edu} % E-mail address of First Author

%\Address{$^\dag$ Department of Quantitative 
%Methods, Universidad Loyola Andaluc\'ia, 
%E-41704 Seville, Spain
%} 
%\URLaddressD{
%\href{http://www.rscosan.com}
%{http://www.rscosan.com}
%}
%\EmailD{rscosa@gmail.com} 

%\Address{$^\ddag$ Department of Mathematics, University of Louisiana at Lafayette, Lafayette, LA 70503, USA
%} 
%\URLaddressD{
%\href{https://userweb.ucs.louisiana.edu/~xxw6637/}
%{https://userweb.ucs.louisiana.edu/$\sim$xxw6637/}
%}
%\EmailD{xswang@louisiana.edu} 
%\Address{$^{\S\S}$ Department of Mathematics,
%University of Rochester, Rochester, NY 146º27, USA
%% Address of First Author
%}
%\EmailD{random9483@gmail.com} % E-mail address of First Author

\ArticleDates{Received~\today~in final form ????; 
Published online ????}
%%%%%%%%%%%%%%%%%%%%%%%%%%%%%%%%%%%%%%%%%%%%%%%%%%%%
\Abstract{In this paper we evaluate integrals of products of gamma functions of Ramanujan type in terms of bilateral hypergeometric series. In cases where the bilateral hypergeometric series are summable, then we evaluate these integral as beta integrals. In addition, we obtain integral representations for bilateral hypergeometric series.
}
%``Symmetry, Integrability and Geometry: Methods and Applications''.}

%Please type here List of Keywords for your article separated by semicolon.
% Keywords:

\Classification{33D45, 05A15, 42C05, 05E05, 33D15}
%{\rcoro *}
%%%%31C12, 32Q45, 33C05, 33C45, 35A08, 35J05, 42A16
%{??????} % e.g. 35A30; 81Q05
%For 2010 Mathematics Subject Classification see
%http://www.ams.org/mathscinet/msc/msc2010.html
%%%%%%%%%%%%%%%%%%%%%%%%%%%%%%%%%%%%%%%%%%%%%%%%%%%%%%%%%%%%%%%%%%%
%\section{Introduction}
%\label{Intro}

% \poro{Here is some motivation.}
%\begin{equation}
%{\overset{\raisebox{0.5ex}{$t_1,t_2$}}{\raisebox{-0.55cm}
%{\resizebox{0.3cm}{!}{\scalebox{0.4}[1.0]{\mbox{$\mathbb{I}$}}}}}}
%\end{equation}
%%%%%%%%%%%%%%%%%%%%%%%%%%%%%%%%%%%%%%%%%%%%%%%%%%%%%%%%%%%%%%%%%%%

%\begin{flushright}
%\begin{minipage}{70mm}
%\it Dedicated to the life and mathematics\\ of Dick Askey, 1933-2019.
%\end{minipage}
%\end{flushright}
%%%%%%%%%%%%%%%%%%%%%%%%%%%%%%%%%%%%%%%%%%%%%%%%%%%%%%%
%\newpage

%\tableofcontents
%\addtocontents{toc}{\protect\color{black}}

%\newgeometry{margin=2cm}
\vspace{-0.10cm}

\begin{flushright}
\begin{minipage}{70mm}
\it Dedicated to George Andrews and Bruce Berndt for their 85\textsuperscript{th} birthdays.
\end{minipage}
\end{flushright}

\section{Introduction}\label{sec1}

In 1920 Ramanujan \cite[(7.1)]{Ramanujan1920} showed that
\begin{equation}\label{intro1}
\int_{-\infty}^\infty\frac{\dd x}{\Gamma(a+x)\Gamma(b+x)\Gamma(c-x)\Gamma(d-x)}
=\frac{\Gamma(a+b+c+d-3)}{\Gamma(a+c-1)\Gamma(a+d-1)\Gamma(b+c-1)\Gamma(b+d-1)}
\end{equation}
provided that $\Re(a+b+c+d)>3$.
This result is listed as Ramanujan's beta integral in the Digital Library of Mathematical Functions 
\cite[\href{http://dlmf.nist.gov/5.13.E4}{(5.13.4)}]{NIST:DLMF}.
It may be compared to Barnes' beta integral \cite[\href{http://dlmf.nist.gov/5.13.E3}{(5.13.3)}]{NIST:DLMF}
\begin{equation}\label{intro2}
\frac1{2\pi} \int_{-\infty}^\infty \Gamma(a+ix)\Gamma(b+ix)\Gamma(c-ix)\Gamma(d-ix)\,\dd x
=\frac{\Gamma(a+c)\Gamma(a+d)\Gamma(b+c)\Gamma(b+d)}{\Gamma(a+b+c+d)}
\end{equation}
which holds for $\Re a, \Re b, \Re c, \Re d>0$.
We notice an obvious difference between these integrals. The integrand in \eqref{intro1} is an entire function of $x$, while the integrand in \eqref{intro2} is 
meromorphic on $\C$ having many poles. The integral \eqref{intro2} may be evaluated by the residue theorem while
this theorem is not applicable to obtain \eqref{intro1}.
In fact, Ramanujan used the Fourier transform as the main tool in his proof of \eqref{intro1}.

In this paper we will investigate more general integrals of Ramanujan's type, namely,
\begin{equation}\label{intro3}
 \int_{-\infty}^\infty \frac{\expe^{-ixt}\,\dd x}{\prod_{j=1}^m \Gamma(a_j+1+x)\Gamma(b_j+1-x)},
\end{equation}
where $t\in\R$, $a_1,\dots,a_m$, $b_1,\dots, b_m$ are given complex numbers such that the integral exists.
Note that $m$ can be any positive integer.
Of course, we cannot expect that such integrals can always be evaluated explicitly in terms of $\Gamma$-functions
as in \eqref{intro1} or \eqref{intro2}.
However, applying the Poisson summation formula from the theory of Fourier transforms we show in Theorem \ref{4:t2} 
that the integral \eqref{intro3} can always be expressed in terms of a finite sum of bilateral hypergeometric series.
For the definition of the bilateral hypergeometric series and its basic properties we refer to Section \ref{sec2}.
In some cases we obtain hypergeometric series that can be evaluated in terms of $\Gamma$-functions.
In such cases we can also evaluate the integral explicitly. We study such cases in Section~\ref{sec5}.

As a second application of the Poisson summation formula we obtain in Theorem \ref{4:t3} an integral representation
of any bilateral hypergeometric series in terms of integrals of the type \eqref{intro3} (the numerator of the integrand will be
a trigonometric polynomial).

In the final section 6 we discuss possible extensions of our results to bilateral basic hypergeometric series.
However, the results for basic bilateral series are not completely parallel to those for bilateral hypergeometric series.

\section{Preliminaries on bilateral hypergeometric series}\label{sec2}

\subsection{Bilateral hypergeometric series}

Recall the definition of the shifted factorial (Pochhammer symbol)
\[ (c)_n:=\frac{\Gamma(c+n)}{\Gamma(c)}\quad\text{for $n\in\Z$}. \]
The bilateral hypergeometric series ${}_pH_q$ \cite[(6.1.1.2)]{Slater66} is defined by
\begin{equation}\label{H}
\Hhyp{p}{q}{c_1,\dots,c_p}{d_1,\dots, d_q}{z}=\sum_{n=-\infty}^\infty \frac{(c_1)_n\dots(c_p)_n}{(d_1)_n\dots (d_q)_n} z^n ,
 \end{equation}
where $p,q\in\N_0$ and $c_1,\dots, c_p, d_1,\dots,d_q\in\C$.
The series \eqref{H} is a formal Laurent series in the variable $z$. It is well-defined if $d_1,\dots,d_q\notin-\N_0$ and $c_1,\dots,c_p\notin\N$.
The shifted factorials satisfy
\[ (c)_n(1-c)_{-n}=(-1)^n\quad\text{for $n\in\Z$}.\]
This yields the symmetry relation
\begin{equation}\label{id}
 \Hhyp{p}{q}{c_1,\dots,c_p}{d_1,\dots, d_q}{z}=
\Hhyp{q}{p}{1-d_1,\dots,1-d_q}{1-c_1,\dots, 1-c_p}{\frac{\epsilon}{z}},\quad\text{where $\epsilon:=(-1)^{p-q}$.}
\end{equation}

The unilateral hypergeometric series ${}_pF_q$ defined by
\[ \hyp{p}{q}{a_1,,\dots,a_p}{b_1,\dots, b_q}{z}=\sum_{n=0}^\infty \frac{(a_1)_n\dots(a_p)_n}{(b_1)_n\dots (b_q)_n} 
\frac{z^n}{n!} \]
can be seen as a special case of the bilateral hypergeometric series.
If $b_q=1$ then
\begin{equation}\label{hyp1} 
 \Hhyp{p}{q}{a_1,\dots,a_p}{b_1,\dots, b_{q-1},1}{z}=
\hyp{p}{q-1}{a_1,\dots,a_p}{b_1,\dots, b_{q-1}}{z}.
\end{equation}
Conversely, the bilateral hypergeometric series can be expressed in terms of the 
unilateral hypergeometric series as follows:
\begin{multline*}
 \Hhyp{p}{q}{c_1,\dots,c_p}{d_1,\dots,d_q}{z}\\
= \hyp{p+1}{q}{1,c_1,\dots,c_p}{d_1,\dots,d_q}{z}
+\hyp{q+1}{p}{1,1-d_1,\dots,1-d_q}{1-c_1,\dots, 1-c_p}{\frac{\epsilon}{z}}\\
= \hyp{p+1}{q}{1,c_1,\dots,c_p}{d_1,\dots,d_q}{z}+
\frac{(1-d_1)\dots(1-d_q)}{(1-c_1)\dots(1-c_p)}\frac{\epsilon}{z}\, 
\hyp{q+1}{p}{1,2-d_1,\dots,2-d_q}{2-c_1,\dots, 2-c_p}{\frac{\epsilon}{z}}.
\end{multline*}
Note that these are identities for formal series.

If at least one of the $c_j$ is in $-\N_0$ then the series \eqref{H} terminates to the right. If at least one of the $d_j$ is in $\N$ then the series
terminates to the left. For example, if $c_1=-m$ , $m\in\N_0$, then 
\[
 \Hhyp{p}{q}{c_1,\dots,c_p}{d_1,\dots,d_q}{z}=\sum_{n=-\infty}^m \frac{(c_1)_n\dots(c_p)_n}{(d_1)_n\dots (d_q)_n} z^n .
 \]
In this case we can even allow that some of the $d_j$ are in $-\N_0$ as long as $d_j\notin\{0,\dots,-m+1\}$ for all $j=1,\dots,q$.

Every Laurent series has an inner radius $r$ and an outer radius $R$ of convergence.
The outer radius of the series \eqref{H} is $R=\infty$ if $p<q$ or if the series terminates to the right.
Otherwise $R=1$ if $p=q$ and $R=0$ if $p>q$. If follows from \eqref{id} that the inner radius of \eqref{H} is $r=0$ if $p>q$ or if the series terminates to the left.
Otherwise, $r=1$ if $p=q$ and $r=\infty$ if $p<q$.

Unless the series terminates, we have $r=R=\infty$ if $p<q$, $r=R=0$ if $p>q$ and $r=R=1$ if $p=q$.
Therefore, if $p\ne q$ and the series does not terminate, then the series \eqref{H} diverges for every $z$.

Now suppose that $p=q$. Then $r=R=1$ unless the series terminates.
Using the ratio test we see that we have absolute convergence on the unit circle $|z|=1$ provided that
\[ \sum_{j=1}^p \Re(c_j-d_j) <-1 .\]
Alternatively, we may use that
\[ \frac{(c_j)_n}{(d_j)_n}=O\left(|n|^{\Re(c_j-d_j)}\right)\quad\text{as $|n|\to\infty$} .\]
Using partial summation we see that we have conditional convergence for $|z|=1$, $z\ne 1$, if
\[   \sum_{j=1}^p \Re(c_j-d_j) <0 .\]
Conditional convergence means that the subseries $\sum_{n=0}^\infty$ and $\sum_{n=-\infty}^{-1}$ 
both converge conditionally.

\subsection{Basic hypergeometric series}

For $q\in\CC$, $|q|<1$, $a\in\C$ and $n\in\N_0$ the basic shifted factorials are defined by
\[ (a;q)_n:= \prod_{k=0}^{n-1} (1-aq^k) \text{ for $n\in\N$,}\]
and
\begin{equation}\label{factorials}
 (a;q)_{-n}:=\prod_{k=1}^n \frac{1}{1-aq^{-k}}=\frac{q^{\frac12n(n+1)}}{(-a)^n (qa^{-1};q)_n} \text{ for $n\in\N$,}
\end{equation}
and
\[
(a;q)_\infty:=\lim_{n\to\infty}(a;q)_n.
\]
The $q$-gamma function is given by \cite[\href{http://dlmf.nist.gov/5.18.E4}{(5.18.4)}]{NIST:DLMF}
\[
\Gamma_q(x):=\frac{(q;q)_\infty(1-q)^{1-x}}{(q^x;q)_\infty},
\]
We know \cite[\href{http://dlmf.nist.gov/5.18.E10}{(5.18.10)}]{NIST:DLMF} that 
\begin{equation}\label{gammalimit}
\lim_{q\to1^-} \Gamma_q(x)= \Gamma(x) .
\end{equation}
\\
Let $q\in\C$, $|q|<1$, and 
let $a_1,\dots,a_m$, $b_1,\dots,b_m\in\C$, where $m$ is any positive integer.
Then the bilateral basic hypergeometric series is defined by
\cite[\href{http://dlmf.nist.gov/17.4.E3}{(17.4.3)}]{NIST:DLMF}
\begin{equation}\label{psi}
 \qpsihyp{m}{m}{a_1,\dots,a_m}{b_1,\dots,b_m}{q,z}=\sum_{n\in \Z} \frac{(a_1,\ldots,a_m;q)_n}{(b_1,\ldots,b_m;q)_n} z^n ,
 \end{equation}
where we used the notation $(a_1,\ldots,a_m;q)_n=\prod_{j=1}^m (a_j;q)_n$.
We can also define the series ${}_r\psi_s$ for $r\ne s$
according to \cite[\href{http://dlmf.nist.gov/17.4.E3}{(17.4.3)}]{NIST:DLMF}.
The series \eqref{psi} is well-defined if $b_j\notin\{q^k: k\in\N_0\}$ and $a_j\notin \{q^{-k}: k\in\N\}$
for all $j=1,\dots,m$.
The series converges absolutely when 
\[ \frac{|b_1\cdots b_m|}{|a_1\cdots a_m|}<|z|<1 .\]
For additional properties of bilateral basic hypergeometric series we refer to Chapter 5 of \cite{GaspRah}.

\subsection{The limit \texorpdfstring{$q\to1^-$}{q->1}}

We have 
\[ \lim_{q\to1^-} \frac{(q^\alpha;q)_n}{(q^\beta;q)_n}= \frac{(\alpha)_n}{(\beta)_n} \quad\text{for $n\in\Z$}.
\]
Therefore, as $q\to1^-$, the basic bilateral hypergeometric series 
\begin{equation}\label{2:psim}
 \qpsihyp{m}{m}{q^{\alpha_1},\dots,q^{\alpha_m}}{q^{\beta_1},\dots,{q^{\beta_m}}}{q,z}
\end{equation}
converges termwise to the bilateral hypergeometric series
\begin{equation}\label{2:Hm}
 \Hhyp{m}{m}{\alpha_1,\dots,\alpha_m}{\beta_1,\dots,\beta_m}{z} .
\end{equation}
However, this does not imply that \eqref{2:psim} converges to \eqref{2:Hm} for some $z\in\C$.
Actually, this statement is meaningless because in the non-terminating case \eqref{2:Hm} is defined only on the unit circle
$|z|=1$ while \eqref{2:psim} is undefined there.
In the following, we show in what sense \eqref{2:Hm} is the limit of \eqref{2:psim} as $q\to1^-$.
We start with some lemmas.

\begin{lem}\label{2:l1}
Let $\alpha=s+it$ with $s>0$, $t\in \R$.
There is a constant $K$ independent of $q$ and $n$ such that
\begin{equation}\label{2:ineq1}
 |(q^\alpha;q)_n|\le K (q^s;q)_n \quad\text{for all $q\in(0,1)$ and $n\in\N_0$}.
\end{equation}
\end{lem}
\begin{proof}
Set $q=e^{-u}$, $u>0$. Then we have 
\[
\frac{|1-q^\alpha|^2}{(1-q^s)^2}=\frac{1+e^{-2us}-2e^{-us}\cos(ut)}{1+e^{-2us}-2e^{-us}}=
1+\frac{\sin^2(\frac12ut)}{\sinh^2(\frac12 us)} .
\]
Therefore,
\[ \frac{|1-q^\alpha|^2}{(1-q^s)^2}\le 1+\frac{t^2}{s^2} .\]
If we apply this inequality with $\alpha+k$ in place of $\alpha$ for $k=0,\dots,n-1$ we get
\[ \frac{|(q^\alpha;q)_n|^2}{(q^s;q)_n^2}\le \prod_{k=0}^{n-1} \left(1+\frac{t^2}{(s+k)^2}\right)\le K^2:=\prod_{k=0}^\infty \left(1+\frac{t^2}{(s+k)^2}\right).\]
This gives \eqref{2:ineq1}.
\end{proof}

\begin{lem}\label{2:l2}
Let $0<\beta\le \alpha$. Then 
\begin{equation}\label{2:ineq2}
 \frac{(q^\alpha;q)_n}{(q^\beta;q)_n}\le \frac{(\alpha)_n}{(\beta)_n}\quad\text{for $q\in(0,1)$, $n\in\N_0$}.
 \end{equation}
In particular, there is a constant $K$ independent of $q$ and $n$ such that
\[ \frac{(q^\alpha;q)_n}{(q^\beta;q)_n}\le K (1+n)^{\alpha-\beta} \quad\text{for $q\in(0,1)$, $n\in\N_0$}.\]
\end{lem}
\begin{proof}
The inequalities $0<\beta\le \alpha$ imply that the function 
$f(q)=\frac{1-q^\alpha}{1-q^\beta}$
is non-decreasing so \eqref{2:ineq2} follows.
\end{proof}

Note that the inequality sign in \eqref{2:ineq2} must be reversed when $0<\alpha\le \beta$. 

\begin{lem}\label{2:l3}
Let $0<\alpha<\beta$ and $\tau>0$. There is a constant $K$ independent of $q$ and $n$ such that
\begin{equation}\label{2:ineq3}
 \frac{(q^\alpha;q)_n}{(q^\beta;q)_n}q^{n\tau} \le K(1+n)^{\alpha-\beta}\quad\text{for $q\in(0,1)$, $n\in\N_0$}.
 \end{equation}
\end{lem}
\begin{proof}
Let $p$ be the smallest positive integer such that $\beta\le \alpha+p$.
For fixed $n\ge p$ and $q\in(0,1)$ we introduce the function
\[ f(z):=\frac{(q^\alpha;q)_n}{(q^z;q)_n} .\]
This function is analytic in the half-plane $\Re z>0$.
For $x>0$ let 
\[ M(x):=\max\{|f(x+iy)|: y\in\R\}= f(x) .\]
By Hadamard's three lines theorem \cite[Ch.\ VI, Thm 3.7]{Conway1978}, $\log  f(x)$ is a convex function.
Since $f(\alpha)=1$ we get
\begin{equation}\label{2:hadamard}
 f(\beta)\le (f(\alpha+p))^{\frac{\beta-\alpha}{p}} .
\end{equation}
Now
\[ f(\alpha+p)=\frac{\prod_{j=0}^{p-1} (1-q^{\alpha+j})}{\prod_{j=0}^{p-1} (1-q^{\alpha+n+j})}
\le \frac{(1-q^{\alpha+p-1})^p}{(1-q^{\alpha+n})^p} .\]
Let $q=e^{-u}$, $u>0$. Then the inequality $1-e^{-x}\le x$ for $x\in\R$ implies
\begin{equation}\label{2:ineq3a}
 f(\alpha+p)\le \left(\frac{(\alpha+p-1)u}{1-e^{-u(\alpha+n)}}\right)^p .
\end{equation}
Set $\sigma:=\frac{p\tau}{\beta-\alpha}$. Then we have 
\begin{equation}\label{2:ineq3b}
 \frac{u}{1-e^{-u(\alpha+n)}}e^{-un p^{-1}\sigma} \le \frac{K_1}{\alpha+n},
 \end{equation}
where
\[ K_1:=\max\{ \frac{v}{1-e^{-v}} e^{-\frac{v\sigma}{\alpha+p}}:v>0\} .\]
Now \eqref{2:ineq3a}, \eqref{2:ineq3b} give
\[ f(\alpha+p) q^{n\sigma} \le (\alpha+p-1)^pK_1^p \frac{1}{(\alpha+n)^p} \quad\text{for $q\in(0,1)$, $n\ge p$.}\]
Using \eqref{2:hadamard} we obtain \eqref{2:ineq3}.
\end{proof}

\begin{lem}\label{2:l4}
Let $\alpha, \beta\in\C$, $-\beta\notin \N_0$.
Let $\tau>0$. Then there is $q_0\in(0,1)$ independent of $n$ and 
a constant $K$ independent of $q$ and $n$ such that
\[ \frac{|(q^\alpha;q)_n|}{|(q^\beta;q)_n|} q^{n\tau}\le K (1+n)^{-\Re(\beta-\alpha)}\quad\text{for all $q\in[q_0,1)$, $n\in\N_0$}.\]
\end{lem}
\begin{proof}
Choose $k\in\N$ so large that $\Re(\alpha+k)>0$ and $\Re(\beta+k)>0$.
There is $q_0\in(0,1)$ and a constant $L$ independent of $q$ and $n$ such that  
\[ \frac{|(q^\alpha;q)_n|}{|(q^\beta;q)_n|} q^{n\tau}\le L q^{n\tau}\prod_{j=k}^{n-1} \frac{|1-q^{\alpha+j}|}{|1-q^{\beta+j}|}.
\quad\text{for $q\in[q_0,1)$, $n\ge k$}.\]
By Lemma \ref{2:l1},
\[  \frac{|(q^\alpha;q)_n|}{|(q^\beta;q)_n|} q^{n\tau}\le K_1 Lq^{n\tau} \prod_{j=k}^{n-1} \frac{1-q^{\Re(\alpha+j)}}{1-q^{\Re(\beta+j)}}
\quad\text{for $q\in[q_0,1)$, $n\ge k$}.\]
We complete the proof by using Lemma \ref{2:l2} if $\Re\alpha\ge \Re\beta$ and Lemma \ref{2:l3} if $\Re\alpha<\Re\beta$.
\end{proof}

We now establish the following convergence theorem connecting \eqref{2:psim} with \eqref{2:Hm}.

\begin{thm}\label{2:t1}
Let $\alpha_j, \beta_j\in\C$ for $j=1,\dots,m$, where $m\in\N_0$.
Suppose that $\alpha_j\notin \N$ and $-\beta_j\notin \N_0$ for all $j=1,\dots,m$,
and that $\Re \sigma>1$, where $\sigma:=\sum_{j=1}^m (\beta_j-\alpha_j)$.
Let $0<\tau<\Re\sigma$. Then
\begin{equation}\label{2:limit}
 \lim_{q\to1^-} \qpsihyp{m}{m}{q^{\alpha_1},\dots,q^{\alpha_m}}{q^{\beta_1},\dots,q^{\beta_m}}{q, q^\tau z}
= \Hhyp{m}{m}{\alpha_1,\dots,\alpha_m}{\beta_1,\dots,\beta_m}{z}
\end{equation}
uniformly on the unit circle $|z|=1$.
\end{thm}
\begin{proof}
We have termwise convergence of the series. 
By Lemma \ref{2:l4}, there is a constant $K$ independent of $q$ and $n$ such that 
\begin{equation}\label{2:ineq}
  q^{n\tau}\prod_{j=1}^m \frac{|(q^{\alpha_j};q)_n|}{|(q^{\beta_j};q)_n|} \le K (1+|n|)^{-\Re\sigma}\quad\text{for $q\in[q_0,1)$, $n\in\N_0$} .
\end{equation}
Using \eqref{factorials} we find, for $n\in\N$,
\[   q^{-n\tau}\prod_{j=1}^m \frac{(q^{\alpha_j};q)_{-n}}{(q^{\beta_j};q)_{-n}} =q^{n(\sigma-\tau)}
\prod_{j=1}^m \frac{(q^{1-\beta_j};q)_n}{(q^{1-\alpha_j};q)_n} .\]
Applying Lemma \ref{2:l4} a second time we show that \eqref{2:ineq} holds for all $n\in\Z$.
Since $\Re\sigma>1$, Tannery's theorem implies \eqref{2:limit}.
\end{proof}

\section{The Poisson summation formula}\label{sec3}

Let $f:\R\to\C$ be integrable, and let 
\[ F(t):=\int_{-\infty}^\infty f(x)\,\expe^{-ixt}\,\dd x,\quad t\in\R,\]
be its Fourier transform. 
The functions $f(x)$ and $F(t)$ are connected through the Poisson summation formula \cite[Thm 3.2.8]{Grafakos2014} 
that we will apply in the following form.

\begin{thm}\label{3:Poisson}
Let $f:\R\to\C$  be a continuous function which satisfies 
\begin{equation} 
|f(x)|\le K (1+|x|)^{-1-\epsilon}\quad\text{for all  $x\in\R$},
\end{equation}
for some positive constants $K,\epsilon$.
Let $F(t)$ be the Fourier transform of $f(x)$.  
Let $\omega>0$ and $t\in\R$ be such that $\sum_{n\in\Z} |F(t+n\omega)|<\infty$. Then
\begin{equation}\label{eq1}
 \sum_{n\in\Z} F(t+n\omega)=\frac{2\pi}{\omega}\sum_{n\in\Z} f\left(\frac{2\pi n}{\omega}\right) \expe^{-i\frac{2\pi n}{\omega}t} .
 \end{equation}
\end{thm}
\begin{proof}
In \cite[Def.\ 2.2.8]{Grafakos2014} the Fourier transform is defined in a slightly different way:
\[ \tilde F(s)=\int_{-\infty}^\infty f(x)\,\expe^{-2\pi  ixs}\,\dd x. \]
Then the Poisson summation formula in its simplest form  \cite[Thm 3.2.8]{Grafakos2014} 
is 
\[ \sum_{n\in \Z} \tilde F(n)=\sum_{n\in \Z} f(n). \]
Replacing $f(x)$ by $f(\frac{2\pi x}{\omega}) \expe^{-2\pi i \frac{xt}{\omega}}$ gives \eqref{eq1}.
\end{proof}

\section{Main results}\label{sec4}

Let $a_1,\dots,a_m$, $b_1,\dots,b_m$ be complex numbers for some $m\in\N$.
We define entire functions
\[ f_j(x):=\frac{1}{\Gamma(a_j+1+x)\Gamma(b_j+1-x)} \quad\text{for $j=1,2,\dots,m$,}
\]
and consider their product
\begin{equation}\label{4:f}
f(x):=\prod_{j=1}^m f_j(x)=\frac{1}{\prod_{j=1}^m\Gamma(1+a_j+x)\Gamma(1+b_j-x)} . 
\end{equation}
We have $f_j(x)={\mathcal O}(|x|^{-\Re(a_j+b_j)-1})$ as $x\in\R$, $|x|\to\infty$.
We assume that $\Re(a_j+b_j)>0$ for all $j=1,\dots,m$ so that $f_j\in L^1(\R)$ for all $j$.
Let 
\[ F_j(t):=\int_{-\infty}^\infty f_j(x)\,\expe^{-ixt}\,\dd x \]
be the Fourier transform of $f_j(x)$.
To compute it Ramanujan \cite[(1.1)]{Ramanujan1920} starts with Cauchy's integral \cite[p.\ 158, (5)]{Nielsen1906} stating that
\begin{equation}\label{4:Cauchy}
\int_{-\frac12\pi}^{\frac12\pi} (\cos t)^{\gamma}\expe^{i\delta t}\,\dd t =\frac{\pi\Gamma(\gamma+1)}{2^\gamma \Gamma(1+\frac12(\gamma+\delta))\Gamma(1+\frac12(\gamma-\delta))}\quad\text{for $\Re \gamma>-1$} .
\end{equation}
It follows from \eqref{4:Cauchy} with $\gamma=a_j+b_j$, $\delta=a_j-b_j+2x$ and Fourier inversion \cite[Thm 2.2.14]{Grafakos2014} that
\begin{equation}\label{4:fourierFj}
F_j(t)=\begin{cases} \displaystyle{\frac{(2\cos(\frac12t))^{a_j+b_j}}{\Gamma(a_j+b_j+1)}\expe^{-\tfrac12 it(b_j-a_j)}} & \text{if $t\in\R$, $|t|\le \pi$}, \\0 & \text{if $t\in\R$, $|t|>\pi$.}\end{cases}\\
\end{equation}
Let $F(t)$ be the Fourier transform of $f(x)$:
\[ F(t):=\int_{-\infty}^\infty f(x)\,\expe^{-itx}\,\dd x .\]
Then
\begin{equation}\label{4:fourierF}
 F(t)=(2\pi)^{1-m} (F_1\ast F_2\ast\dots\ast F_m)(t),
 \end{equation}
 where the convolution is defined by
\[ (F\ast G)(t)=\int_{-\infty}^\infty F(s)G(t-s)\,\dd s .\]

\begin{lem}\label{4:l1}
Let $\sum_{j=1}^m \Re(a_j+b_j+1)>1$.
Then $F(t)=0$ for $t\in\R$ with $|t|\ge m\pi$.
\end{lem}
\begin{proof}
Assume first that $\Re(a_j+b_j)>0$ for all $j=1,\dots,m$. 
By \eqref{4:fourierFj}, $F_j(t)=0$ for $t\in\R$, $|t|\ge \pi$.
Now $F(t)=0$ for $t\in\R$ with $|t|\ge m\pi$ follows from \eqref{4:fourierF}.
By analytic continuation, the statement of the lemma is also true 
under the assumption $\sum_{j=1}^m \Re(a_j+b_j+1)>1$.
\end{proof}

\begin{thm}\label{4:t2}
Let $\sum_{j=1}^m \Re(a_j+b_j+1)>1$.
Suppose that $p\in\N$, $p\ge m$ and $t\in\R$ with $|t|\le p\pi$. Then we have 
\begin{equation}\label{eq2}
F(t)=\int_{-\infty}^\infty f(x)\,\expe^{-ixt}\,\dd x=\sum_{k=0}^p S_k(t),
\end{equation}
where $f(x)$ is defined in \eqref{4:f},  
\begin{equation}
 S_k(t):=\frac1p\sum_{\ell=-\infty}^\infty f(\ell+\tfrac{k}{p}) \expe^{-i(\ell+\frac{k}{p}) t }
=\frac1p C_k \expe^{-i\frac{k}{p}t}\Hhyp{m}{m}{-b_1+\frac{k}{p},\dots,-b_m+\frac{k}{p}}{a_1+1+\frac{k}{p},\dots,a_m+1+\frac{k}{p} }{(-1)^m\expe^{-it}},
\end{equation}
and
\begin{equation} C_k:=\frac{1}{\prod_{j=1}^m \Gamma(a_j+1+\frac{k}{p})\Gamma(b_j+1-\frac{k}{p})} .
\end{equation}
\end{thm}
\begin{proof}
Theorem \ref{3:Poisson} with $\omega=2\pi p$ gives
\[ \sum_{n\in\Z} F(t+2\pi p n)=\frac1p\sum_{n\in\Z} f\left(\frac{n}{p}\right) \expe^{-i\frac{n}{p}t} .\]
By Lemma \ref{4:l1}, the left-hand sum reduces to $F(t)$.
By setting $n=p\ell+k$, $k=0,\dots,p-1$, we find
\[ F(t)=\frac1p\sum_{n\in\Z} f\left(\frac{n}{p}\right)\,\expe^{-i\frac{n}{p}t}=\sum_{k=0}^{p-1} S_k(t) .\]
We note that
\begin{align*}
 f_j(\ell+\tfrac{k}{p})&=\frac{1}{\Gamma(a_j+1+\ell+\tfrac{k}{p})\Gamma(b_j+1-\ell-\tfrac{k}{p})}\\
 &= \frac{(-1)^\ell}{\Gamma(a_j+1+\tfrac{k}{p})\Gamma(b_j+1-\tfrac{k}{p})} 
 \frac{(\tfrac{k}{p}-b_j)_\ell}{(a_j+1+\tfrac{k}{p})_\ell}.
\end{align*}
If we multiply theses identities for $j=1,\dots,m$, multiply by $\expe^{-i (\ell +\frac{k}{p}) t}$ and then add for $\ell\in\Z$ we obtain the 
second formula for $S_k(t)$.
\end{proof}

\begin{lem}\label{l2}
We have $S_k(t)=S_{k+p}(t)$. If $a_j=b_j$ for all $j=1,\dots,m$ then $S_k(t)=S_{-k}(-t)=S_{p-k}(-t)$.
\end{lem}
\begin{proof}
Change $\ell$ to $\ell+1$ or $-\ell$ in the definition of $S_k$.
\end{proof}

Note that the sum $\sum_{k=0}^{p-1} S_k(t)$ appearing on the right hand side of \eqref{eq2} is just a Riemann sum with step size $1/p$ for the integral on the left-hand side.
Usually, these Riemann sums are only approximations of the integral but in our case they are equal to the integral.
The Poisson summation formula also gives us an integral representation of bilateral hypergeometric functions as follows.

\begin{thm}\label{4:t3}
Let $\sum_{j=1}^m \Re(a_j+b_j+1)>1$.
If $t\in[-\pi,\pi]$ then
\begin{equation} C_0\,
\Hhyp{m}{m}{-b_1,\dots,-b_m}{a_1+1,\dots,a_m+1}{-\expe^{-it}}=\int_{-\infty}^\infty f(x)\,\expe^{-ixt}g_m(x)\,\dd x
,\end{equation}
where
\begin{equation} C_0:= \frac{1}{\prod_{j=1}^m \Gamma(a_j+1)\Gamma(b_j+1)}
\end{equation}
and
\begin{equation} g_m(x):=\frac{\sin(m\pi x)}{\sin(\pi x)}=\begin{cases} 1+2\sum_{n=1}^{\frac12(m-1)} \cos(2n\pi x) & \text{if $m$ is odd,}\\
2\sum_{n=1}^{\frac12m} \cos((2n-1)\pi x) & \text{if $m$ is even.} \end{cases}
\end{equation}
\end{thm}
\begin{proof}
This time we apply Theorem \ref{3:Poisson} with $\omega=2\pi$. 
Suppose first that $t\in[-\pi,\pi]$ and $m$ is odd.
Then the points $t+2n\pi$ lie outside $(-m\pi,m\pi)$ and so $F(t+2n\pi)=0$ unless $n=-\frac12(m-1),\dots,\frac12(m-1)$. 
So we find
\[ \sum_{n\in\Z} f(n) \expe^{-i nt} =\sum_{n\in\Z} F(t+2\pi n)=\sum_{n=-\frac12(m-1)}^{\frac12(m-1)} F(t+2\pi n)
=\int_{-\infty}^\infty f(x)\,\expe^{-itx}g_m(x)\,\dd x .\]
As in the proof of Theorem \ref{4:t2} we see that
\[  \sum_{n\in\Z} f(n) \expe^{-i nt}=m S_0(t)\]
and the desired statement follows.

Now let $s\in[0,2\pi]$, and let $m$ be even.
Then the points $s+2n\pi$ lie outside $(-m\pi,m\pi)$ unless $n=-\frac12m,\dots,\frac12m-1$. 
Now Theorem \ref{3:Poisson} yields
\[ \sum_{n\in\Z} f(n) \expe^{-i ns} =\sum_{n=-\frac12 m}^{\frac12m-1} F(s+2\pi n)
=\int_{-\infty}^\infty f(x)\,\expe^{-isx} \expe^{i\pi x}g_m(x)\,\dd x \]
and this is the claimed equation after substitution $s=t+\pi$.
\end{proof}

\begin{thm}\label{4:t4}
Let $\sum_{j=1}^m \Re(a_j+b_j+1)>1$.
Then
\begin{align}
\int_{-\infty}^\infty f(x)\frac{\sin((m-1)\pi x)}{\sin(\pi x)}\,\dd x&=
 C_0\,
\Hhyp{m}{m}{-b_1,\dots,-b_m}{a_1+1,\dots,a_m+1}{1},
 \\
\int_{-\infty}^\infty f(x)\frac{\sin(m\pi x)}{\sin(\pi x)}\,\dd x&=
 C_0\,
\Hhyp{m}{m}{ -b_1,\dots,-b_m}{a_1+1,\dots,a_m+1}{-1}.
\end{align}
\end{thm}
\begin{proof}
This follows from Theorem \ref{4:t3} with $t=\pi$ and $t=0$. If $t=\pi$ we are also using Lemma \ref{4:l1}.
\end{proof}

\section{Special cases}\label{sec5}
\subsection{The case \texorpdfstring{$m=1$}{m=1}}

If we combine \eqref{4:fourierFj} 
with Theorem \ref{4:t2} for $p=m=1$ then we find
\[ 
\Hhyp{1}{1}{-b}{a+1}{-\expe^{-it}}
=\frac{\Gamma(a+1)\Gamma(b+1)}{\Gamma(a+b+1)} (2\cos\tfrac12t)^{a+b} \expe^{-\tfrac12 it(b-a)}
\quad\text{for $t\in[-\pi,\pi]$}.
\]
The above equation is equivalent to 
\begin{equation*}
\Hhyp11{a}{b}{-\expe^{-it}}=
\frac{\Gamma(1-a)\Gamma(b)}{\Gamma(b-a)}
\expe^{\frac12 it(a+b-1)}(2\cos(\tfrac12 t))^{b-a-1},
\end{equation*}
where $t\in[-\pi,\pi]$, $a\not\in\N$, $b\not\in-\N_0$.
One is able to derive the above equation and the following related equation
\begin{equation*}
\Hhyp11{a}{b}{\expe^{it}}=
\frac{\Gamma(1-a)\Gamma(b)}{\Gamma(b-a)}
\expe^{\frac12 i(\pi-t)(a+b-1)}(2\sin(\tfrac12 t))^{b-a-1}\quad\text{for $t\in[0,2\pi]$,}
\end{equation*}
with the same constraints on $t$ by starting with Ramanujan's formula \cite[\href{http://dlmf.nist.gov/17.8.E2}{(17.8.2)}]{NIST:DLMF}
\begin{equation}\label{5:1psi1}
\qpsihyp11{q^a}{q^b}{q,z}=\frac{(q,q^{b-a},q^az,q^{1-a}/z;q)_\infty}{(q^b,q^{1-a},z,q^{b-a}/z;q)_\infty}
=\frac{\Gamma_q(b)\Gamma_q(1-a)}{\Gamma_q(b-a)} \frac{(q^az,q^{1-a}/z;q)_\infty}{(z,q^{b-a}/z;q)_\infty},
\end{equation}
where $q^{\Re(b-a)}<|z|<1$. 
We choose $0<\tau<\Re(b-a)$, set $z=-q^\tau\expe^{-it}$ and $z=q^\tau\expe^{it}$ in \eqref{5:1psi1}, respectively, and let $q\to1^-$.
On the left-hand side we apply Theorem \ref{2:t1}. On the right-hand side we use \eqref{gammalimit} and the 
limit
\begin{equation}\label{5:limit}
 \lim_{q\to 1^-} \frac{(q^\alpha z;q)_\infty}{(q^\beta z;q)_\infty}=(1-z)^{\beta-\alpha}
\quad \text{for $0<|z|\le 1$.}
\end{equation}
The limit \eqref{5:limit} can be derived from the $q$-binomial theorem \cite[\href{http://dlmf.nist.gov/17.2.E37}{(17.2.37)}]{NIST:DLMF}; see also \cite[Theorem 10.2.4]{AAR}. 
Two consequences of the above results are 
\begin{equation*}
\Hhyp11{a}{b}{-1}=
2^{b-a-1}\frac{\Gamma(1-a)\Gamma(b)}{\Gamma(b-a)},
\end{equation*}
and ${}_1H_1(a;b;1)=0$.

\subsection{The case \texorpdfstring{$m=2$}{m=2}}

It follows from Theorem \ref{4:t4} that
\begin{equation}\label{5:eqm2}
 \int_{-\infty}^\infty \frac{\dd x}{\prod_{j=1}^2 \Gamma(a_j+1+x)\Gamma(b_j+1-x)}
=  C_0\, 
\Hhyp{2}{2}{ -b_1,-b_2 }{a_1+1,a_2+1}{1},
\end{equation}
where
\[ C_0=\frac{1}{\prod_{j=1}^2 \Gamma(a_j+1)\Gamma(b_j+1)} .\]
We have the generalized Gauss theorem \cite[(6.1.2.1)]{Slater66}
\begin{multline}\label{5:2H2}
\Hhyp{2}{2}{a,b}{c,d}{1}
 =\Gamma\left[\begin{array}{c} c,d, 1-a,1-b,c+d-a-b-1\\ c-a,d-a,c-b,d-b\end{array}\right]\\
 :=\frac{\Gamma(c)\Gamma(d)\Gamma(1-a)\Gamma(1-b)\Gamma(c+d-a-b-1)}{\Gamma(c-a)\Gamma(d-a)\Gamma(c-b)\Gamma(d-b)},
\end{multline}
valid for $\Re(c+d-a-b-1)>0$.
Note that the above formula generalizes the Gauss formula \cite[\href{http://dlmf.nist.gov/15.4.E20}{(15.4.20)}]{NIST:DLMF} with a choice of $d=1$ 
using \eqref{hyp1}.
Using \eqref{5:2H2} in \eqref{5:eqm2} we find Ramanujan's integral \cite[(7.1)]{Ramanujan1920}
\begin{equation}\label{5:Ram1}
\hspace{-0.2cm}\int_{-\infty}^\infty \frac{\dd x}{\prod_{j=1}^2 \Gamma(a_j\!+\!1\!+\!x)\Gamma(b_j\!+\!1\!-\!x)}
  = \frac{\Gamma(a_1+b_1+a_2+b_2+1)}
{\Gamma(a_1+b_1+1)\Gamma(a_1+b_2+1)\Gamma(a_2+b_1+1)\Gamma(a_2+b_2+1)}.
\end{equation}
It is interesting to note that Ramanujan obtained \eqref{5:Ram1} directly from \eqref{4:fourierFj}, \eqref{4:fourierF} so we have a proof of \eqref{5:2H2}.

It follows from Theorem \ref{4:t4} that 
\begin{equation}\label{Ram2}
 \int_{-\infty}^\infty \frac{2\cos(\pi x)\,\dd x}{\prod_{j=1}^2 \Gamma(a_j+1+x)\Gamma(b_j+1-x)}
=  C_0\, 
\Hhyp{2}{2}{ -b_1,-b_2}{a_1+1,a_2+1 }{-1}.
\end{equation}
In general, this equation cannot be simplified further because we do not have a summation formula for the ${}_2H_2$ series at $z=-1$.
However, Ramanujan \cite[(7.2)]{Ramanujan1920} used \eqref{4:fourierFj}, \eqref{4:fourierF} to show that under the assumption  
\begin{equation}\label{5:ass} 
a_1-b_1=a_2-b_2
\end{equation}
we have
\begin{equation*}
\int_{-\infty}^\infty \frac{\expe^{-i\pi x}\,\dd x}{\prod_{j=1}^2 \Gamma(a_j+1+x)\Gamma(b_j+1-x)}
=\frac{\expe^{-\tfrac12 i\pi(b_1-a_1)}}{2\Gamma(\tfrac12(a_1+b_1)+1)\Gamma(\tfrac12(a_2+b_2)+1)\Gamma(a_1+b_2+1)} .
\end{equation*}
By comparing this result with \eqref{Ram2} we obtain
\begin{equation*}
\Hhyp{2}{2}{-b_1,-b_2}{a_1+1,a_2+1}{-1}
 =\cos(\tfrac12(b_1-a_1)\pi) \frac{\Gamma(a_1+1)\Gamma(b_1+1)\Gamma(a_2+1)\Gamma(b_2+1)}
{\Gamma(\tfrac12(a_1+b_1)+1)\Gamma(\tfrac12(a_2+b_2)+1)\Gamma(a_1+b_2+1)} 
\end{equation*}
provided \eqref{5:ass} holds.
 
\subsection{The case \texorpdfstring{$m=3$}{m=3}}

It follows from Theorem \ref{4:t4} that
\[ \int_{-\infty}^\infty \frac{2\cos(\pi x)\,\dd x}{\prod_{j=1}^3 \Gamma(1+a_j+x)\Gamma(1+b_j-x)}
= C_0\, 
\Hhyp{3}{3}{ -b_1,-b_2,-b_3 }{a_1+1,a_2+1,a_3+1}{1},
%{}_3 H_3\left(\begin{array}{c} -b_1,-b_2,-b_3 \\ a_1+1,a_2+1,a_3+1\end{array}; 1\right),
\]
where
\[ C_0=\frac{1}{\prod_{j=1}^3 \Gamma(a_j+1)\Gamma(b_j+1)} .\]
According to \cite[(6.1.2.6)]{Slater66} we have the following formula for a well-poised ${}_3H_3(1)$,
\begin{eqnarray}
&&\hspace{-0.5cm}
\Hhyp{3}{3}{b,c,d}{1+a-b,1+a-c,1+a-d}{1}
%{}_3H_3\left(\begin{array}{ccc} b,&c,&d\\ 1+a-b,&1+a-c,&1+a-d\end{array};1\right)
\nonumber\\
&&= \Gamma\left[\begin{array}{c} 1-b,1-c,1-d,1+a-b,1+a-c,1+a-d,1+\tfrac12a,1-\tfrac12a,1+\tfrac32a-b-c-d\\
1+a-c-d,1+a-b-d,1+a-b-c,1+\tfrac12a-b,1+\tfrac12a-c,1+\tfrac12a-d,1+a,1-a\end{array}
\right],\label{5:3H3}
\end{eqnarray}
for $\Re(1+\tfrac32a-b-c-d)>0$. Therefore, setting $b_1=-b$, $b_2=-c$, $b_3=-d$, we obtain the following result.

\begin{thm}\label{5:tm3a}
Let  $\Re(1+\frac32a+b_1+b_2+b_3)>0$. Then 
\begin{equation}\label{5:m3a}
 \int_{-\infty}^\infty \frac{2\cos(\pi x)\,\dd x}{\prod_{j=1}^3 \Gamma(b_j+1+a+x)\Gamma(b_j+1-x)}
=\frac{\cos(\tfrac12\pi a)\,\Gamma(1+\tfrac32a+b_1+b_2+b_3)}{\prod_{j=1}^3 \Gamma(1+\tfrac12 a+b_j)\prod_{1\le i<j\le 3} \Gamma(1+a+b_i+b_j)}.
\end{equation}
\end{thm}

At this point an interesting question arises. Can we prove Theorem \ref{5:tm3a} based on \eqref{4:fourierFj}, \eqref{4:fourierF}  without using \eqref{5:3H3}?
If this is possible we have a new proof of \eqref{5:3H3}.
We can prove Theorem \ref{5:tm3a} with $a=0$ without using \eqref{5:3H3} if we can show the following equation:
\begin{align*}
&\int_{-\pi}^\pi \int_{-s_1}^\pi (2\cos\tfrac12s_1)^{2b_1}(2\cos\tfrac12s_2)^{2b_2}(2\sin(\tfrac12(s_1+s_2))^{2b_3}\,\dd s_2\,\dd s_1\\
&=\frac{2\pi^2 \Gamma(2b_1+1)\Gamma(2b_2+1)\Gamma(2b_3+1)\Gamma(b_1+b_2+b_3+1)}{\Gamma(b_1+1)\Gamma(b_2+1)\Gamma(b_3+1)\Gamma(b_2+b_2+1)\Gamma(b_1+b_3+1)
\Gamma(b_2+b_3+1)} .
\end{align*}

It follows from Theorem \ref{4:t4} that
\[ \int_{-\infty}^\infty \frac{(1+2\cos(2\pi x))\,\dd x}{\prod_{j=1}^3 \Gamma(1+a_j+x)\Gamma(1+b_j-x)}
= C_0\, 
\Hhyp{3}{3}{-b_1,-b_2,-b_3}{a_1+1,a_2+1,a_3+1}{-1},
%{}_3 H_3\left(\begin{array}{c} -b_1,-b_2,-b_3 \\ a_1+1,a_2+1,a_3+1\end{array}; -1\right),
\]
where
\[ C_0=\frac{1}{\prod_{j=1}^3 \Gamma(a_j+1)\Gamma(b_j+1)} .\]
Set $a_j=c_j+\frac14$, $b_j=c_j-\frac14$ and substitute $x=y-\frac14$. Then 
\[ \frac{1+2\cos(2\pi x)}{\prod_{j=1}^3 \Gamma(1+a_j+x)\Gamma(1+b_j-x)}
=\frac{1+2\sin(2\pi y)}{\prod_{j=1}^3 \Gamma(1+c_j+y)\Gamma(1+c_j-y)},
\]
so we obtain the following result using the observation that
$\int_{-\infty}^\infty f(x)\,\dd x=0$ when $f(x)$ is an odd function.

\begin{thm}\label{5:tm3b}
Let $\Re(1+c_1+c_2+c_3)>0$. Then
\begin{equation}\label{5:m3b}
\int_{-\infty}^\infty \frac{\dd y}{\prod_{j=1}^3 \Gamma(1+c_j+y)\Gamma(1+c_j-y)}=
C\,
\Hhyp{3}{3}{\frac14-c_1,\frac14-c_2,\frac14-c_3}{\frac54+c_1,\frac54+c_2,\frac54+c_3}{-1},
%{}_3H_3\left(\begin{array}{c} \frac14-c_1,\frac14-c_2,\frac14-c_3 \\ \frac54+c_1,\frac54+c_2,\frac54+c_3 \end{array};-1\right),
\end{equation}
where
\begin{equation} 
C=\frac{1}{\prod_{j=1}^3 \Gamma(c_j+\tfrac54)\Gamma(c_j+\tfrac34)} .
\end{equation}
\end{thm}

\subsection{The case \texorpdfstring{$m=4$}{m=4}}

It follows from Theorem \ref{4:t4} that
\[
 \int_{-\infty}^\infty \frac{(1+2\cos(2\pi x))\,\dd x}{\prod_{j=1}^4 \Gamma(1+a_j+x)\Gamma(1+b_j-x)}
= C_0\,  
\Hhyp{4}{4}{-b_1,\dots,-b_4}{a_1+1,\dots,a_4+1}{1},
%{}_4 H_4\left(\begin{array}{c} -b_1,\dots,-b_4 \\ a_1+1,\dots,a_4+1\end{array}; 1\right),
\]
where
\[ C_0=\frac{1}{\prod_{j=1}^4 \Gamma(a_j+1)\Gamma(b_j+1)} .\]
When we set $a_j=c_j+\frac14$, $b_j=c_j-\frac14$ and substitute $x=y-\tfrac14$ we find the following result.

\begin{thm}\label{5:tm4a}
Let $\Re(c_1+c_2+c_3+c_4)>-\tfrac32$. Then
\begin{equation} \label{5:m4a}
\int_{-\infty}^\infty \frac{\dd y}{\prod_{j=1}^4 \Gamma(1+c_j+y)\Gamma(1+c_j-y)}
=C\,  
\Hhyp{4}{4}{ \frac14-c_1,\dots,\frac14-c_4 }{c_1+\tfrac54,\dots,c_4+\tfrac54}{1},
%{}_4 H_4\left(\begin{array}{c} \frac14-c_1,\dots,\frac14-c_4 \\ c_1+\tfrac54,\dots,c_4+\tfrac54\end{array}; 1\right),
\end{equation}
where
\begin{equation} C=\frac{1}{\prod_{j=1}^4 \Gamma(c_j+\tfrac54)\Gamma(c_j+\tfrac34)} .
\end{equation}
\end{thm}

Consider Bailey's bilateral very-well-poised ${}_6\psi_6$ summation \cite[\href{http://dlmf.nist.gov/17.8.E7}{(17.8.7)}]{NIST:DLMF}
\begin{equation}\label{5:6psi6}
\qpsihyp66{\pm qa^\frac12,b,c,d,e}{\pm a^\frac12,\frac{qa}{b},\frac{qa}{c},\frac{qa}{d},\frac{qa}{e}}{q,\frac{qa^2}{bcde}}
=\frac{(q,qa,\frac{q}{a},\frac{qa}{bc},\frac{qa}{bd},\frac{qa}{be},\frac{qa}{cd},\frac{qa}{ce},\frac{qa}{de};q)_\infty}
{(\frac{q}{b},\frac{q}{c},\frac{q}{d},\frac{q}{e},\frac{qa}{b},\frac{qa}{c},\frac{qa}{d},\frac{qa}{e},\frac{qa^2}{bcde};q)_\infty},
\end{equation}
where $|qa^2|<|bcde|$.
Setting $e=-a^{1/2}$ and replacing $a,b,c,d$ by $q^a, q^b, q^c, q^d$, respectively, and taking the limit $q\to1^-$
using Theorem \ref{2:t1}, \eqref{gammalimit} and \eqref{5:limit}, we obtain the following summation formula for a very-well-poised ${}_4H_4(-1)$, namely 
\begin{equation}\label{5:4H4m1sum}
\Hhyp44{1\!+\!\frac12 a,b,c,d}{\frac12 a,1\!+\!a\!-\!b,1\!+\!a\!-\!c,1\!+\!a\!-\!d}{-1}=
\Gamma\left[\begin{array}{c} 1-b, 1-c,1-d,1+a-b, 1+a-c, 1+a-d\\ 1-a, 1+a, 1+a-b-c, 1+a-b-d, 1+a-c-d
\end{array}\right].
\end{equation}
If we combine this with Theorem \ref{4:t4}, we obtain the following result.

\begin{thm}\label{5:tm4b}
Let $\Re(3a+2b_1+2b_2+2b_3)>-1$. Then
\begin{eqnarray}
&& \hspace{-0.6cm}\int_{-\infty}^\infty \frac{(2\cos(\pi x)+2\cos(3\pi x))\,\dd x}
{\Gamma(\tfrac12a+x)\Gamma(-\tfrac12 a-x)\prod_{j=1}^3 \Gamma(1+a+b_j+x)\Gamma(1+b_j-x)}\nonumber\\
&&=\frac{1}{\Gamma(\tfrac12a)\Gamma(-\tfrac12a)   
\Gamma(1-a)\Gamma(1+a)\Gamma(1+a+b_1+b_2)\Gamma(1+a+b_1+b_3)\Gamma(1+a+b_2+b_3)}.\label{5:m4b}
\end{eqnarray}
\end{thm}

In \eqref{5:m4b} we substitute $x=y-\tfrac12 a$ ,
$b_j=c_j-\tfrac12 a$. Then we obtain
\begin{multline}\label{5:m4c}
\int_{-\infty}^\infty \frac{(2\cos(\pi y)\cos(\tfrac12\pi a)+2\cos(3\pi y)\cos(\tfrac32\pi a))\, \dd y}{\Gamma(y)\Gamma(-y)\prod_{j=1}^3 (1+c_j+y)\Gamma(1+c_j-y)}\\
=\frac{1}{\Gamma(\tfrac12a)\Gamma(-\tfrac12a)   
\Gamma(1-a)\Gamma(1+a)\Gamma(1+c_1+c_2)\Gamma(1+c_1+c_3)\Gamma(1+c_2+c_3)}
\end{multline}
valid for $\Re(c_1+c_2+c_3)>-\frac12$.

\subsection{The case \texorpdfstring{$m=5$}{m=5}}

It follows from Theorem \ref{4:t4} that
\[ \int_{-\infty}^\infty \frac{(2\cos(\pi x)+2\cos(3\pi x))\,\dd x}{\prod_{j=1}^5 \Gamma(1+a_j+x)\Gamma(1+b_j-x)}
= C_0\,  
\Hhyp{5}{5}{-b_1,\dots,-b_5}{a_1+1,\dots,a_5+1}{1},
%{}_5 H_5\left(\begin{array}{c} -b_1,\dots,-b_5 \\ a_1+1,\dots,a_5+1\end{array}; 1\right),
\]
where
\[ C_0=\frac{1}{\prod_{j=1}^5 \Gamma(a_j+1)\Gamma(b_j+1)} .\]
According to \cite[(6.1.2.5)]{Slater66}  one has the following evaluation of a very-well-poised ${}_5H_5(1)$, namely
\begin{eqnarray}
&& 
\Hhyp{5}{5}{1+\tfrac12a,b,c,d,e}{\tfrac12a,1+a-b,1+a-c,1+a-d,1+a-e}{1}
%{}_5H_5\left(\begin{array}{ccccc} 1+\tfrac12a,& b,&c,&d,&e\\ \tfrac12a,&1+a-b,&1+a-c,&1+a-d,&1+a-e\end{array};1\right)
\nonumber \\
&&= \Gamma\left[\begin{array}{c} 1-b,1-c,1-d,1-e,1+a-b,1+a-c,1+a-d,1+a-e,1+2a-b-c-d-e\\\nonumber
1+a, 1-a, 1+a-b-c,1+a-b-d,1+a-b-e,1+a-c-d,1+a-c-e,1+a-d-e\end{array}
\right]\label{5:5H5}
\end{eqnarray}
for $\Re(1+2a-b-c-d-e)>0$. 
If we combine these results we obtain the following theorem.

\begin{thm}\label{5:tm5a}
Let $\Re(1+2a+b_1+b_2+b_3+b_4)>0$. Then
\begin{eqnarray}
&&\hspace{-3cm}\int_{-\infty}^\infty \frac{(2\cos(\pi x)+2\cos(3\pi x))\,\dd x}{\Gamma(\frac12a+x)\Gamma(-\frac12 a-x)\prod_{j=1}^4 \Gamma(1+a+b_j+x)\Gamma(1+b_j-x)}\nonumber\\
&&\hspace{-1cm}=\frac{\Gamma(1+2a+b_1+b_2+b_3+b_4)}{\Gamma(\frac12a)\Gamma(-\frac12a)\Gamma(1-a)\Gamma(1+a)\prod_{1\le i<j\le 4} \Gamma(1+a+b_i+b_j)} .\label{5:m5a}
\end{eqnarray}
\end{thm}

If we substitute $x=y-\frac12 a$, $b_j=c_j-\tfrac12a$ in \eqref{5:m5a}, then we obtain
\begin{multline}\label{eqm5a2}
\int_{-\infty}^\infty \frac{(2\cos(\pi y)\cos(\tfrac12\pi a)+2\cos(3\pi y)\cos(\tfrac32\pi a))\,\dd y}{\Gamma(y)\Gamma(-y)\prod_{j=1}^4 \Gamma(1+c_j+y)\Gamma(1+c_j-y)}\\
=\frac{\Gamma(1+c_1+c_2+c_3+c_4)}{\Gamma(\frac12a)\Gamma(-\frac12a)\Gamma(1-a)\Gamma(1+a)\prod_{1\le i<j\le 4} \Gamma(1+c_i+c_j)}
\end{multline}
valid for $\Re(1+c_1+c_2+c_3+c_4)>0$. 
Taking $a=\frac13$ gives us the following result.

\begin{thm}\label{5:tm5b}
Let $\Re(1+c_1+c_2+c_3+c_4)>0$. Then
\begin{equation}\label{5:m5b}
\int_{-\infty}^\infty \frac{\cos(\pi y)\,\dd y}{\Gamma(y)\Gamma(-y)\prod_{j=1}^4 \Gamma(1+c_j+y)\Gamma(1+c_j-y)}
=-\frac{1}{8\pi^2}\frac{\Gamma(1+c_1+c_2+c_3+c_4)}{\prod_{1\le i<j\le 4} \Gamma(1+c_i+c_j)} .
\end{equation}
\end{thm}

Note that 
\[ \frac{4\cos(\pi y)}{\Gamma(y)\Gamma(-y)}=\frac{1}{\Gamma(2y)\Gamma(-2y)}\]
so that Theorem \ref{5:tm5b} agrees with
\cite[Thm 4.6]{CohlVolkmer2024}.

\subsection{The case \texorpdfstring{$m=6$}{m=6}}

\begin{thm}\label{5:m6}
Let $\Re(a_1+a_2+a_3+a_4+a_5+a_6)>-\tfrac52$. Then
\begin{equation} \int_{-\infty}^\infty \frac{\dd x}{\prod_{j=1}^6 \Gamma(1+a_j+x)\Gamma(1+a_j-x)}
= 2 S_0(0)+4S_2(0)=2 S_3(0)+4S_1(0) ,
\end{equation}
where $S_k$ is defined as in Theorem \ref{4:t2} for $m=p=6$ and $a_j=b_j$ for $j=1,\dots,6$.
\end{thm}
\begin{proof}
By Theorem \ref{4:t2} we have
\[ F(t):=\int_{-\infty}^\infty \frac{\expe^{-ixt}\, \dd x}{\prod_{j=1}^6 \Gamma(1+a_j+x)\Gamma(1+a_j-x)}
= \sum_{j=0}^5 S_j(t) \]
for $t\in[-6\pi,6\pi]$. Since $F(6\pi)=0$ we get 
\[ S_0(0)+S_2(0)+S_4(0)=S_1(0)+S_3(0)+S_5(0) .\]
Moreover, by Lemma \ref{l2}, $S_2(0)=S_4(0)$ and $S_1(0)=S_5(0)$.
This completes the proof.
\end{proof}

Let $\Re(a_1+a_2+a_3+a_4)>-1$ and set $a_5:=-1$, $a_6:=-\tfrac12$.
By the special choice of $a_5$, $a_6$, we have $S_0(0)=S_3(0)=0$.
Since
\[ \frac{1}{\Gamma(2x)\Gamma(-2x)}=\frac{4\pi}{\Gamma(x)\Gamma(-x)\Gamma(\tfrac12+x)\Gamma(\tfrac12-x)} ,\]
Theorem \ref{5:m6} implies
\[ \int_{-\infty}^\infty \frac{\dd x}{\Gamma(2x)\Gamma(-2x)\prod_{j=1}^4 \Gamma(a_j+1+x)\Gamma(a_j+1-x)}=16\pi S_2(0).\]
By Theorem \ref{4:t2},
\[ S_2(0)= \tfrac16 C_2\, 
\Hhyp{6}{6}{\tfrac13-a_1,\tfrac13-a_2,\tfrac13-a_3,\tfrac13-a_4, \tfrac43,\tfrac56}{a_1+\tfrac43, a_2+\tfrac43,a_3+\tfrac43, a_4 +\tfrac43, \tfrac13, \tfrac56}{1}
%{}_6H_6\left(\begin{array}{c} \\
%\end{array};1\right) 
.\]
This ${}_6H_6$-series reduces to a ${}_5H_5$-series. We evaluate the ${}_5H_5$-series by \eqref{5:5H5} and obtain again \cite[Thm 4.6]{CohlVolkmer2024}.

\section{A Fourier transform and \texorpdfstring{$q$}{q}-extensions}\label{sec6}

\subsection{A generalization of Ramanujan's Fourier transform}

Ramanujan \cite{Ramanujan1920} used the Fourier transform
\begin{equation}\label{6:int0}
\int_{-\infty}^\infty \frac{\expe^{-ixt}\,\dd x}{\Gamma(\alpha+x)\Gamma(\beta-x)}
=\begin{cases} \displaystyle{\frac{(2\cos(\frac12t))^{\alpha+\beta-2}}{\Gamma(\alpha+\beta-1)}\expe^{-\tfrac12 it(\beta-\alpha)}} & \text{if $t\in\R$, $|t|\le \pi$}, \\0 & \text{if $t\in\R$, $|t|>\pi$,}\end{cases}
\end{equation}
where $\alpha,\beta$ are complex parameters such that $\Re(\alpha+\beta)>2$ (one can also allow $\Re(\alpha+\beta)>1$ but then the integrand might not be 
in $L^1(\R)$ anymore).
Our goal is to derive an analogue of this Fourier transform in the $q$-world.

\begin{lem}\label{6:Fl1}
For $q\in\C$, $0<|q|<1$, $c\in\C$, $c\ne 0$, define the entire function
\[ g(x):=(cq^{-x};q)_\infty q^{\frac12x(x+1)} c^{-x} ,\]
where the powers denote their principal values.
Then $g(x)$ is bounded for $x\ge 0$, and 
\begin{equation}\label{6:limit}
\lim_{n\to\infty}(-1)^n g(s+n)=g(s)(c^{-1}q^{1+s};q)_\infty \quad\text{for $s\in\C$}.
\end{equation}
\end{lem}
\begin{proof}
Let $s\in\C$ and $n\in\N_0$. Then 
\begin{eqnarray*}
 g(s+n)&=&(cq^{-s-n};q)_\infty q^{\frac12(s+n)(s+n+1)} c^{-s-n}\\
 &=& \frac{(cq^{-s};q)_\infty}{(cq^{-s};q)_{-n}}  q^{\frac12(s+n)(s+n+1)}  c^{-s-n}\\
 &=& (cq^{-s};q)_\infty (c^{-1}q^{1+s};q)_n q^{-\frac12n(n+1)}  q^{\frac12(s+n)(s+n+1)}(-cq^{-s})^n  c^{-s-n}\\
 &=&(-1)^n g(s) (c^{-1}q^{1+s};q)_n .
\end{eqnarray*}
This establishes \eqref{6:limit}. The function $g(x)$ is bounded for $x\ge 0$ because
$g(s)$ is bounded on $[0,1]$ and $ (c^{-1}q^{1+s};q)_n$ is bounded for $s\in[0,1]$ and $n\in\N_0$.
\end{proof}

Let $q\in\C$, $0<|q|<1$, and $a,b,w\in\C$, $a,b,w\ne 0$. We define the entire function
\[ f(x;w;q):=(bq^x;q)_\infty (a^{-1}q^{1-x};q)_\infty q^{\frac12x(x-1)} w^x.\]

\begin{lem}\label{6:Fl2}
There are constants $K,L$ (independent of $x$) such that
\[ |f(x;w;q)|\le K \left|\frac{w}{a}\right|^x \quad\text{for $x\ge 0$},\]
and
\[ |f(x;w;q)|\le L \left|\frac{w}{b}\right|^x \quad\text{for $x\le 0$}.\]
\end{lem}
\begin{proof}
This follows from Lemma \ref{6:Fl1}.
\end{proof}

Lemma \ref{6:Fl2} shows that $f\in L^1(\R)$ when
\begin{equation}\label{6:ass1}
 |b|<|w|< |a| .
 \end{equation}

\begin{thm}\label{6:t1}
Under assumption \eqref{6:ass1} we have 
\begin{equation}\label{6:int1}
 \int_{-\infty}^\infty f(x;w;q)\,\dd x =\frac{(\frac{b}{a};q)_\infty}{(-\frac{w}{a},-\frac{b}{w};q)_\infty}
 \frac{\sqrt{2\pi w}\exp\left(\frac{(\log  w)^2}{2\log  q^{-1}}\right)} {q^{1/8}\sqrt{\log  q^{-1}}}.
 \end{equation} 
\end{thm}
\begin{proof}
Similar to \cite[\S 3]{IsmailRahman1995} we find
\begin{eqnarray}
&&\hspace{-3.3cm}\int_{-\infty}^\infty f(x;w;q)\,\dd x= \int_0^1 \sum_{k=-\infty}^\infty f(x+k;w;q)\, \dd x \nonumber\\
&&\hspace{-0.3cm}= \frac{(q,\frac{b}{a};q)_\infty}{(-\frac{w}{a},-\frac{b}{w}
;q)_\infty} \int_0^1 (-wq^x,w^{-1}q^{1-x};q)_\infty q^{\frac12x(x-1)}w^x\,\dd x .\label{6:Feq2}
\end{eqnarray}
To prove this equation Ramanujan's summation formula \eqref{5:1psi1} is used.
If we apply \eqref{6:Feq2} with $b\to0$ and $a\to\infty$ we obtain
\begin{equation}\label{6:eq3}
 \int_{-\infty}^\infty q^{\frac12x(x-1)} w^x\,\dd x
=(q;q)_\infty \int_0^1 (-wq^x,-w^{-1}q^{1-x};q)_\infty q^{\frac12x(x-1)} w^x\,\dd x .
\end{equation}
Note that \eqref{6:eq3} can also be derived from 
\[ \int_{-\infty}^\infty h(x)\,\dd x =\int_0^1 \sum_{k=-\infty}^\infty h(x+k)\,{\dd x} \]
for $h(x)=q^{\frac12x(x-1)} w^x$ and Jacobi's triple product identity \cite[Thm 10.4.1]{AAR}.
The integral on the left-hand side of \eqref{6:eq3} 
can be evaluated:
\begin{equation}\label{6:eq4}
 \int_{-\infty}^\infty q^{\frac12x(x-1)} w^x\,\dd x=\frac{\sqrt{2\pi w}\exp\left(\frac{(\log  w)^2}{2\log  q^{-1}}\right)} {q^{1/8}\sqrt{\log  q^{-1}}}.
 \end{equation}
Combining \eqref{6:Feq2}, \eqref{6:eq3}, \eqref{6:eq4} we obtain \eqref{6:int1}.
\end{proof}

\begin{thm}\label{6:Ft2}
Under assumption \eqref{6:ass1} we have the Fourier transform 
\begin{equation}\label{6:int2}
 \int_{-\infty}^\infty f(x;w;q)\,\expe^{-ixt}\,\dd x=
\frac{(\frac{b}{a};q)_\infty}{(-\frac{w}{a}\expe^{-it},-\frac{b}{w}\expe^{it};q)_\infty}
 \frac{\sqrt{2\pi w}\expe^{-\frac12it}\exp\left(\frac{(\log  w-it)^2}{2\log  q^{-1}}\right)} {q^{1/8}\sqrt{\log  q^{-1}}}
\end{equation}
valid for $t$ in the strip
\[ \log  \frac{|b|}{|w|}<\Im t<\log  \frac{|a|}{|w|} .\] 
\end{thm} 
\begin{proof}
Replace $w$ by $w\expe^{-it}$ in \eqref{6:int1}.
\end{proof}

We notice a crucial difference comparing \eqref{6:int0} with \eqref{6:int2}. In \eqref{6:int0} the Fourier transform is only defined for $t\in\R$
and it has compact support. In \eqref{6:int2} the Fourier transform is defined in a  strip parallel to the real axis and it is an analytic function
there. Therefore, it cannot have compact support.

\subsection{The limit \texorpdfstring{$q\to1^-$}{q->1}}

In this section we demonstrate that \eqref{6:int0} can be obtained from \eqref{6:int2} in the limit $q\to1^-$.
The notation $f(q)\sim g(q)$ means that $f(q)/g(q)\to 1$ as $q\to 1^-$. We also use the notation $q=\expe^{-u}$ for $u>0$.

\begin{lem}\label{6:asy3}
Let $a\in\C\setminus[1,\infty)$. Then
\[  (a;q)_\infty \sim (1-a)^{1/2}\exp(-u^{-1} \Li(a)) ,\]
where $\Li$ denotes the dilogarithm function \cite[\href{http://dlmf.nist.gov/25.12}{\S 25.12}]{NIST:DLMF}.
\end{lem}
\begin{proof}
Define the function
\[ f(x)=\log (1-a\expe^{-xu})\quad\text{for $x\ge 0$}.\]
By the Euler-Maclaurin formula \cite[p.\ 524]{Knopp1990}, we have 
\[ \sum_{k=0}^n f(k)=\int_0^n f(x)\,\dd x +\tfrac12(f(0)+f(n))-\tfrac1{12}(f'(0)-f'(n))+\int_0^n P_3(x) f'''(x)\,\dd x,\]
where $P_3(x)$ is the function with period $1$ determined by
\[ P_3(x)=\tfrac16x^3-\tfrac14 x^2+\tfrac1{12} x \quad \text{for $x\in[0,1]$} .\]
We note that
\[ f'(x)=\frac{au\expe^{-xu}}{1-a\expe^{-xu}},\quad f'''(x)=\frac{au^3\expe^{-xu}(1+a \expe^{-xu})}{(1-a\expe^{-xu})^3} .\]
We let $n\to\infty$, and get
\[ \sum_{k=0}^\infty f(k)=\int_0^\infty f(x)\,\dd x +\tfrac12f(0)-\tfrac1{12}f'(0)+\int_0^\infty P_3(x) f'''(x)\,\dd x.\]
Now
\[ \log  (a;q)_\infty =\sum_{k=0}^\infty f(k), \quad \int_0^\infty f(x)\,\dd x=-u^{-1} \Li(a) ,\]
and
\[f(0)=\log (1-a),\quad f'(0)=\frac{au}{1-a} .\]
Since $|P_3(x)|\le \frac{\sqrt3}{216}$ and $|f'''(x)|\le |a|(1+|a|)M^{-3}u^3\expe^{-xu}$, where 
\[M=\min\{|1-ta|: t\in[0,1]\},\] 
we estimate
\[ \left|\int_0^\infty P_3(x)f'''(x)\,\dd x\right|\le  \frac{\sqrt3}{216}|a|(1+|a|)M^{-3}u^2=:Ku^2 .\]
Therefore, we find
\begin{equation}\label{6:ineq}
 \left|\log  (a;q)_\infty+u^{-1} \Li(a)-\tfrac12\log (1-a)+\tfrac1{12}\frac{au}{1-a}\right|\le  Ku^2.
\end{equation}
It follows that 
\[ \left|(a;q)_\infty \exp(u^{-1} \Li(a))(1-a)^{-1/2} \exp\left(\tfrac1{12}\frac{au}{1-a}\right)-1\right|\le \expe^{K u^2} K u^2 .\]
This implies the statement of the lemma.
\end{proof}

\begin{lem}\label{6:asy4}
Let $a\in\C\setminus[1,\infty)$ and $\alpha\in \C$. Then
\[ (aq^\alpha;q)_\infty \sim (1-a)^{\frac12-\alpha}\exp(-u^{-1} \Li(a)).\]
\end{lem}
\begin{proof}
Use  \eqref{6:ineq} with $aq^\alpha$ in place of $a$ and note that 
\[ \Li(a\expe^{-\alpha u})=\Li(a)+\alpha u \log  (1-a)+O(u^2)\quad \text{as $u\to0$}.\]
\end{proof}

We are now in a position to demonstrate that \eqref{6:int2} is a $q$-analogue of \eqref{6:int0}.
Consider \eqref{6:int2} with $w=1$, $q\in(0,1)$ and  $a=q^{1-\beta}$, $b=q^\alpha$ with $\beta>2$ and $\alpha>1$.
If we divide both sides of \eqref{6:int2} by $(q;q)_\infty^2 (1-q)^{2-\alpha-\beta}$
then the left-hand side is
\[ \int_{-\infty}^\infty f_q(x)\,\expe^{-ixt}\,\dd x,\]
where
\[ f_q(x):=\frac{q^{\frac12x(x-1)}}{\Gamma_q(x+\alpha)\Gamma_q(\beta-x)}.\]
In order to show that 
\[\lim_{q\to 1^-} \int_{-\infty}^\infty f_q(x)\,\expe^{-ixt}\,\dd x =\int_{-\infty}^\infty \frac{\expe^{-ixt}\,\dd x}{\Gamma(x+\alpha)\Gamma(\beta-x)},\]
we apply Lebesgue's dominated convergence theorem. Pointwise convergence of the integrands follows from \eqref{gammalimit},
and the dominating function is supplied by the following lemma.

\begin{lem}\label{6:dominate}
There are constants $K,L$ independent of $q,x$ such that
\begin{eqnarray}
 |f_q(x)|&\le& K (1+x)^{1-\beta}\quad\text{for $q\in[\tfrac12,1)$, $x\ge 0$},\label{6:est1}\\
|f_q(x)|&\le &L (1+|x|)^{-\alpha}\quad \text{for $q\in[\tfrac12,1)$, $x\le 0$}.\label{6:est2}
\end{eqnarray}
\end{lem}
\begin{proof}
The $q$-gamma function satisfies
\[ \Gamma_q(x+1)=\frac{1-q^x}{1-q} \Gamma_q(x) \]
which implies
\[ f_q(x+1)=\frac{q^x-q^{\beta-1}}{1-q^{x+\alpha}} f_q(x) .\]
Using the Bernoulli inequality
\[ (1+y)^r\le 1+ry\quad\text{with $y=q^x-1$, $r=\frac{\beta-1}{x}$},\]
we estimate 
\[ 0\le\frac{q^{\beta-1}-q^x}{1-q^{x+\alpha}}\le \frac{q^{\beta-1}-q^x}{1-q^x}\le \frac{x-\beta+1}{x} \quad\text{for $q\in(0,1)$, $x\ge \beta-1$}.\]
For $n\in\N$, $x\ge \beta-1$, we get
\begin{equation}\label{6:ineq2}
 |f_q(x+n)|\le |f_q(x)|\prod_{k=1}^n \frac{x-\beta+k}{x+k-1}=|f_q(x)|\frac{\Gamma(x)\Gamma(x-\beta+1+n)}{\Gamma(x-\beta+1)\Gamma(x+n)} .
 \end{equation}
 Since $\frac{1}{\Gamma_q(z)}$ converges to $\frac{1}{\Gamma(z)}$ locally uniformly on $\C$ as $q\to1^-$, $|f_q(x)|$ is bounded on $q\in[\tfrac12,1)$, $x\in[0,\beta]$.
Now \eqref{6:ineq2} with $x\in[\beta-1,\beta]$ and $y=x+n$ shows that there is a constant $K_1$ such that
\[ |f_q(y)|\le K_1 \frac{\Gamma(y-\beta+1)}{\Gamma(y)} \quad \text{for $q\in[\tfrac12,1)$, $y\ge \beta$}.\]
This proves \eqref{6:est1}. The proof of \eqref{6:est2} is similar.
\end{proof}

We now wish to determine the limit of the right-hand side of \eqref{6:int2} divided by $(q;q)_\infty^2(1-q)^{2-\alpha-\beta}$, as $q\to 1^-$, that is the limit of
\begin{equation}\label{6:defh}
 h(q):=\frac{(1-q)^{\alpha+\beta-2}}{(q;q)_\infty^2} \frac{(q^{\alpha+\beta-1};q)_\infty}{(-q^{\beta-1}\expe^{-it},-q^\alpha \expe^{it};q)_\infty}
 \frac{\sqrt{2\pi}\expe^{-\frac12it}\exp(-\frac12u^{-1}t^2)} {q^{1/8}\sqrt{u}}
 \end{equation}
as $q\to 1^-$ by direct inspection.
We have 
\begin{equation}\label{6:sim1}
 (q^{\alpha+\beta-1};q)_\infty\sim \frac{(q;q)_\infty(1-q)^{2-\alpha-\beta}}{\Gamma(\alpha+\beta-1)} 
 \end{equation}
by \eqref{gammalimit}, and
\begin{eqnarray*}
 (-q^{\beta-1}\expe^{-it},q)_\infty&\sim &(1+\expe^{-it})^{\frac32-\beta}\exp\left(-u^{-1}\Li(-\expe^{-it})\right),\\
 (-q^\alpha \expe^{it};q)_\infty&\sim &(1+\expe^{it})^{\frac12-\alpha}\exp\left(-u^{-1}\Li(-\expe^{it})\right),
\end{eqnarray*}
by Lemma \ref{6:asy4}.
Note that \cite[\href{http://dlmf.nist.gov/25.12.E7}{(25.12.7)}, \href{http://dlmf.nist.gov/25.12.E8}{(25.12.8)}]{NIST:DLMF} yields
\begin{equation}\label{6:sim2}
 \Li(-\expe^{-it})+\Li(-\expe^{it})=\tfrac12t^2-\tfrac16\pi^2\quad \text{for $t\in[-\pi,\pi]$},
 \end{equation}
and
\[ (1+\expe^{-it})^{\beta-\frac32}(1+\expe^{it})^{\alpha-\frac12}=(2\cos(\tfrac12 t))^{\alpha+\beta-2}\expe^{-\frac12it(\beta-\alpha-1)}.\]
Therefore, for $t\in(-\pi,\pi)$,
\begin{equation}\label{6:sim3}
  \frac{1}{(-q^{\beta-1}\expe^{-it},q)_\infty (-q^\alpha \expe^{it};q)_\infty} \sim  \exp\left(u^{-1}(\tfrac12t^2-\tfrac16\pi^2)\right)
(2\cos(\tfrac12t))^{\alpha+\beta-2} \expe^{-\frac12it(\beta-\alpha-1)} .
\end{equation}
Substituting \eqref{6:sim1}, \eqref{6:sim3} and \cite[\href{http://dlmf.nist.gov/17.2.E6_1}{(17.2.6\_1)}]{NIST:DLMF} in \eqref{6:defh} we obtain that
\[
 h(q)\sim \frac{(2\cos(\frac12t))^{\alpha+\beta-2}}{\Gamma(\alpha+\beta-1)}\expe^{-\tfrac12 it(\beta-\alpha)} 
 \]
 as desired. Using \eqref{gammalimit} in place of Lemma \ref{6:asy4} we obtain $\lim_{q\to1^-}h(q)=0$ for $t=\pm \pi$.
 If $t\in\R$ is outside the interval $[-\pi,\pi]$ then $h(q)$ converges to $0$ as $q\to1^-$. For example, if $t\in(\pi,3\pi)$
 then \eqref{6:sim2} gives 
 \[
 \Li(-\expe^{-it})+\Li(-\expe^{it})=\tfrac12(t-2\pi)^2-\tfrac16\pi^2,
 \]
 and this leads to $\lim_{q\to 1^-} h(q)=0$.

\subsection{Bilateral basic hypergeometric series}

We now follow the line of reasoning in Section \ref{sec4}. For $j=1,\dots,m$ we define entire functions
\[ f_j(x):=(b_jq^x;q)_\infty (a_j^{-1}q^{1-x};q)_\infty q^{\frac12x(x-1)}w_j^x,\]
where $w_1,\dots,w_m\in\C\setminus\{0\}$.
Consider their product 
\[ f(x):=\prod_{j=1}^m f_j(x) ,\]
and let 
\[ F(t):=\int_{-\infty}^\infty f(x)\,\expe^{-ixt}\,\dd x\]
be its Fourier transform.
It follows from Lemma \ref{6:Fl1} that this Fourier transform exists if
\begin{equation}\label{6:ass2}
 |b_1|\cdots|b_m|<|w_1|\cdots|w_m|<|a_1|\cdots |a_m|.
\end{equation}

\begin{thm}\label{6:Ft3}
Let $q\in\C$, $0<|q|<1$, and suppose that $a_1,\dots,a_m\in\C$, $b_1,\dots, b_m\in\C$, $w_1,\dots,w_m\in\C$ 
satisfy \eqref{6:ass2}.
Then, for all $t\in\R$, 
\[ \qpsihyp{m}{m}{a_1,\dots,a_m}{b_1,\dots,b_m}{q,z}=
\lim_{r\to1^-} \int_{-\infty}^\infty f(x)\,\expe^{-ixt}\frac{1-r^2}{1-2r\cos(2\pi x)+r^2} \,\dd x,\]
where 
\[ z:=(-1)^m \expe^{-it}\prod_{j=1}^m \frac{w_j}{a_j} .\]
\end{thm}
\begin{proof}
We note that
\[ (b_jq^n;q)_\infty=\frac{(b_j;q)_\infty}{(b_j;q)_n}\]
and, by \eqref{factorials},
\[ (a_j^{-1}q^{1-n};q)_\infty=\frac{(a_j^{-1}q;q)_\infty (a_j;q)_n}{(-a_j)^n q^{\frac12n(n-1)}}.\]
Therefore, we have
\[ f_j(n)=(b_j;q)_\infty(a_j^{-1}q;q)_\infty\frac{(a_j;q)_n}{(b_j;q)_n}
\left(-\frac{w_j}{a_j}\right)^n ,\]
so
\begin{equation}\label{6:sum1}
 \sum_{n\in\Z} f(n)\,\expe^{-int}=\left(\prod_{j=1}^m (b_j;q)_\infty(a_j^{-1}q;q)_\infty\right)
\qpsihyp{m}{m}{a_1,\dots,a_m}{b_1,\dots,b_m}{q,z} .
\end{equation}
Theorem \ref{3:Poisson} with $\omega=2\pi$ implies
\begin{equation}\label{6:sum2}
 \sum_{n\in\Z} f(n)\,\expe^{-int}=\sum_{n\in\Z} \int_{-\infty}^\infty f(x)\,\expe^{-i(t+2n\pi)x} \,\dd x .
 \end{equation}
Uniform convergence of the series for $r\in[0,1]$ and Lebesgue's dominated convergence imply that
\begin{align*}
 \lim_{r\to1^-} \sum_{n\in\Z}  \int_{-\infty}^\infty f(x)r^{|n|}\expe^{-i(t+2n\pi)x}\,\dd x
&=\sum_{n\in\Z} \lim_{r\to1^-} \int_{-\infty}^\infty f(x)r^{|n|}\expe^{-i(t+2n\pi)x}\,\dd x \\
&=\sum_{n\in\Z}  \int_{-\infty}^\infty f(x)\,\expe^{-i(t+2n\pi)x}\,\dd x.
\end{align*}
Since
\[ \sum_{n\in\Z} \int_{-\infty}^\infty \left|f(x)r^{|n|}\expe^{-i(t+2n\pi)x}\right|\,\dd x
\le \frac{1+r}{1-r}\int_{-\infty}^\infty |f(x)|\,\dd x<\infty\]
we have 
\begin{align*}
 \sum_{n\in\Z}  \int_{-\infty}^\infty f(x)r^{|n|} \expe^{-i(t+2n\pi)x}\,\dd x
&=\int_{-\infty}^\infty \sum_{n\in\Z} f(x)r^{|n|}\expe^{-i(t+2n\pi)x}\,\dd x\\
&=\int_{-\infty}^\infty f(x)\,\expe^{-ixt} \frac{1-r^2}{1-2r\cos(2\pi x)+r^2} \,\dd x,
\end{align*}
so application of \eqref{6:sum1}, \eqref{6:sum2} completes the proof.
\end{proof}

\subsection{Integral representations for a \texorpdfstring{${}_6\psi_6$ and other basic bilateral series}{5:6psi6}}

From \cite[Theorem 4.5]{CohlVolkmer2024}, one has the following $q$-beta integral
\begin{eqnarray}
&&\hspace{-2.4cm}
{\sf I}(\alpha;{\bf a};q):=\int_{-\infty}^\infty
(1+q^{2x}\alpha^2)
(-q^{x+1}\alpha {\mathbf a},\frac{q^{1-x}}{\alpha}{\mathbf a}
;q)_\infty 
q^{2x^2-x}\alpha^{4x}\,\dd x\nonumber\\
&&\hspace{-0.8cm}=\frac{\sqrt{2\pi}\,\alpha \exp\left(\frac{2(\log \alpha)^2}{\log q^{-1}}\right)(-qab,-qac,-qad,-qbc,-qbd,-qcd;q)_\infty}
{q^\frac18\sqrt{\log q^{-1}}(qabcd;q)_\infty},\label{qbetaCohl}
\end{eqnarray}
which is valid for $0<|q|<1$, $\alpha,a,b,c,d\in\C\setminus\{0\}$, ${\bf a}$ be the multiset given by $\{a,b,c,d\}$, and $|abcd|<|q|^{-1}$. It is straightforward to verify by using \eqref{5:6psi6}
that the $q$-beta integral \eqref{qbetaCohl} can be written in terms a basic bilateral very-well-poised ${}_6\psi_6$, namely 
\begin{eqnarray}
&&\hspace{-2cm}{\sf I}(\alpha;{\bf a};q)
=\frac{\sqrt{2\pi}\,\alpha \exp\left(\frac{2(\log \alpha)^2}{\log q^{-1}}\right)(iq^{\frac54}{\bf a},iq^{\frac34}{\bf a};q)_\infty}
{q^\frac18\sqrt{\log q^{-1}}(q,q^\frac12,q^\frac32;q)_\infty}
\qpsihyp66{\pm q^\frac54,-\frac{iq^\frac14}{\bf a}}{\pm q^\frac14,iq^\frac54{\bf a}}{q,qabcd}.
\label{psi66rep}
\end{eqnarray}
Now in order to demonstrate how this is a $q$-extension of a  Ramanujan-type integral, make the replacement 
$\{\alpha,a,b,c,d\}\mapsto\{-iq^\alpha,-iq^a,-iq^b,-iq^c,-iq^d\}$ and rewrite the above integral in terms of $q$-gamma functions, and one obtains
\begin{eqnarray}
&&\hspace{-0.5cm}
\int_{-\infty}^\infty
\frac{\Gamma_q(2(x+\alpha)+1)\,q^{2x^2-x+4\alpha x}\expe^{-2i\pi x}}{\Gamma_q(2(x+\alpha),
1+a\pm(x+\alpha),
1+b\pm(x+\alpha),
1+c\pm(x+\alpha),
1+d\pm(x+\alpha)
)}\,\dd x
\nonumber\\
&&\hspace{-0.1cm}=\frac{-i\sqrt{2\pi}q^\alpha \exp\left(\frac{2(\log (-iq^\alpha))^2}{\log q^{-1}}\right)}
{q^\frac18(1-q)\sqrt{\log q^{-1}}(q;q)_\infty^3}
\frac{\Gamma_q(a+b+c+d+1)}{\Gamma_q(a+b+1,a+c+1,a+d+1,b+c+1,b+d+1,c+d+1)}\nonumber\\
&&\hspace{-0.1cm}=\frac{-i\sqrt{2\pi}q^\alpha \exp\left(\frac{2(\log (-iq^\alpha))^2}{\log q^{-1}}\right)}
{q^\frac18(1-q)\sqrt{\log q^{-1}}(q;q)_\infty^3}
\frac{\Gamma_q(\frac12,\frac32)}{\Gamma_q(\frac54+a,\frac54+b,\frac54+c,\frac54+d,\frac34+a,\frac34+b,\frac34+c,\frac34+d)}\nonumber\\
&&\hspace{6.7cm}\times\qpsihyp66{\pm q^\frac54,q^{\frac14-a},q^{\frac14-b},q^{\frac14-c},q^{\frac14-d}}{\pm q^\frac14,q^{\frac54+a},q^{\frac54+b},q^{\frac54+c},q^{\frac54+d}}{q,q^{a+b+c+d+1}}.
\label{qbetaCohlgam}
\end{eqnarray}
Note that one has the following limit
\begin{equation}
\lim_{q\to 1^{-}}\frac{iq^{\alpha-\frac18}\sqrt{2\pi}
\exp\left(\frac{2(\log(-iq^\alpha))^2}{\log q^{-1}}\right)}{(q;q)_\infty^3(1-q)\sqrt{\log q^{-1}}}
=-\frac{i\expe^{2i\pi\alpha}}{2\pi}.
\label{limit}
\end{equation}
Starting with \eqref{qbetaCohlgam} and
using \eqref{limit}, the $q\to 1^{-}$ limit of the above integral exists and it is given by Theorem \ref{5:tm5b} above.

\medskip
\noindent One might consider limits of \eqref{qbetaCohl} and \eqref{psi66rep}, and ask whether there exist generalizations of these integrals and whether appropriate $q\to1^{-}$ limits exist or not. Let's start with the limit as $d\to 0$ which gives us
\begin{eqnarray}
&&\hspace{-2.4cm}
\int_{-\infty}^\infty
(1+q^{2x}\alpha^2)
(-q^{x+1}\alpha {\mathbf a},\frac{q^{1-x}}{\alpha}{\mathbf a}
;q)_\infty 
q^{2x^2-x}\alpha^{4x}\,\dd x\nonumber\\
&&\hspace{-0.8cm}=\frac{\sqrt{2\pi}\,\alpha \exp\left(\frac{2(\log \alpha)^2}{\log q^{-1}}\right)(-qab,-qac,-qbc;q)_\infty}
{q^\frac18\sqrt{\log q^{-1}}},\nonumber\\
&&\hspace{-0.8cm}=\frac{\sqrt{2\pi}\,\alpha \exp\left(\frac{2(\log \alpha)^2}{\log q^{-1}}\right)(iq^{\frac54}{\bf a},iq^{\frac34}{\bf a};q)_\infty}
{q^\frac18\sqrt{\log q^{-1}}(q,q^\frac12,q^\frac32;q)_\infty}
\qpsihyp56{\pm q^\frac54,-\frac{iq^\frac14}{\bf a}}{\pm q^\frac14,iq^\frac54{\bf a},0}{q,-iq^{\frac54}abc}.
\label{qbetaCohl2}
\end{eqnarray}
which is valid for $0<|q|<1$, $\alpha,a,b,c\in\C\setminus\{0\}$, ${\bf a}$ be the multiset given by $\{a,b,c\}$. Now in order to inquire about whether the $q\to1^{-}$ limit exists, we make the replacement $\{\alpha,a,b,c\}\mapsto\{-iq^\alpha,-iq^a,-iq^b,-iq^c\}$ and
rewrite 
\eqref{qbetaCohl2} in terms of $q$-gamma functions.
\begin{eqnarray}
&&\hspace{-1.0cm}\int_{-\infty}^\infty
\frac{\Gamma_q(2(x+\alpha)+1)q^{2x^2-x+4\alpha x}\expe^{-2i\pi x}}
{\Gamma_q(2(x+\alpha),1+a\pm(x+\alpha),1+b\pm(x+\alpha),1+c\pm(x+\alpha))}
\dd x\nonumber\\
&&\hspace{-0.1cm}=\frac{-i\sqrt{2\pi}q^\alpha \exp\left(\frac{2(\log (-iq^\alpha))^2}{\log q^{-1}}\right)}
{q^\frac18(1-q)\sqrt{\log q^{-1}}(q;q)_\infty^3}
\frac{1}{\Gamma_q(a+b+1,a+c+1,b+c+1)}\nonumber\\
&&\hspace{-0.1cm}=\frac{-i\sqrt{2\pi}q^\alpha \exp\left(\frac{2(\log (-iq^\alpha))^2}{\log q^{-1}}\right)}
{q^\frac18(1-q)\sqrt{\log q^{-1}}(q;q)_\infty^3}
\frac{\Gamma_q(\frac12,\frac32)}{\Gamma_q(\frac54+a,\frac54+b,\frac54+c,\frac34+a,\frac34+b,\frac34+c)}\nonumber\\
&&\hspace{6.7cm}\times\qpsihyp56{\pm q^\frac54,q^{\frac14-a},q^{\frac14-b},q^{\frac14-c}}{\pm q^\frac14,q^{\frac54+a},q^{\frac54+b},q^{\frac54+c},0}{q,q^{\frac54+a+b+c}}.
\label{qbetaCohlgam2}
\end{eqnarray}
Now taking the limit as $q\to 1^{-}$ of \eqref{qbetaCohlgam2} using
\eqref{limit} reduces the right-hand side 
to a very-well-poised ${}_4H_4(-1)$ which 
can be summed using \eqref{5:4H4m1sum}, namely
\begin{eqnarray}
&&\hspace{-1.4cm}\int_{-\infty}^\infty
\frac{\dd x}{\Gamma(\pm 2x,1+a\pm x,1+b\pm x,1+c\pm x)}=-\frac{1}{2\pi^2}\frac{1}{\Gamma(a+b+1,a+c+1,b+c+1)}\nonumber\\
&&\hspace{0.0cm}=-\frac{1}{4\pi}
\frac{1}{\Gamma(\frac54+a,\frac54+b,\frac54+c,\frac34+a,\frac34+b,\frac34+c)}
\Hhyp44{\frac54,\frac14-a,\frac14-b,\frac14-c}{\frac14,\frac54+a,\frac54+b,\frac54+c}{-1}.
\label{H44int}
\end{eqnarray}
We note that \eqref{H44int} agrees with \eqref{5:m4c} when we set $a=\frac13$ in \eqref{5:m4c}. If one takes the limit as $c\to 0$ in \eqref{qbetaCohl2}, then one obtains
\begin{eqnarray}
&&\hspace{-1.0cm}
\int_{-\infty}^\infty
(1+q^{2x}\alpha^2)
(-q^{x+1}\alpha a,-q^{x+1}\alpha b,\frac{q^{1-x}}{\alpha}{a},
\frac{q^{1-x}}{\alpha}{b}
;q)_\infty 
q^{2x^2-x}\alpha^{4x}\,\dd x\nonumber\\
&&\hspace{0.2cm}=\frac{\sqrt{2\pi}\,\alpha \exp\left(\frac{2(\log \alpha)^2}{\log q^{-1}}\right)(-qab;q)_\infty}
{q^\frac18\sqrt{\log q^{-1}}},\nonumber\\
&&\hspace{0.2cm}=\frac{\sqrt{2\pi}\,\alpha \exp\left(\frac{2(\log \alpha)^2}{\log q^{-1}}\right)(iq^{\frac54}{a},
iq^{\frac54}{b},
iq^{\frac34}{a},
iq^{\frac34}{b}
;q)_\infty}
{q^\frac18\sqrt{\log q^{-1}}(q,q^\frac12,q^\frac32;q)_\infty}
\qpsihyp46{\pm q^\frac54,-\frac{iq^\frac14}{a},
\frac{iq^\frac14}{b}
}{\pm q^\frac14,
iq^\frac54{a},
iq^\frac54{b},
0,0}{q,-q^{\frac32}ab}.
\label{qbetaCohl3}
\end{eqnarray}
which is valid for $0<|q|<1$, $\alpha,a,b\in\C\setminus
\{0\}$.
Now in order to inquire about whether the $q\to1^{-}$ limit exists, we
make the replacement $\{\alpha,a,b\}\mapsto\{-iq^\alpha,-iq^a,-iq^b\}$
and rewrite 
\eqref{qbetaCohl3} in terms of $q$-gamma functions, 
\begin{eqnarray}
&&\hspace{0.1cm}\int_{-\infty}^\infty
\frac{\Gamma_q(2(x+\alpha)+1)q^{2x^2-x+4\alpha x}\expe^{-2i\pi x}}
{\Gamma_q(2(x+\alpha),1+a\pm(x+\alpha),1+b\pm(x+\alpha))}
\dd x\nonumber\\
&&\hspace{0.1cm}=\frac{-i\sqrt{2\pi}q^\alpha \exp\left(\frac{2(\log (-iq^\alpha))^2}{\log q^{-1}}\right)}
{q^\frac18(1-q)\sqrt{\log q^{-1}}(q;q)_\infty^3}
\frac{(1-q)^{a+b-1}}{\Gamma_q(a+b+1)}\nonumber\\
&&\hspace{0.1cm}=\frac{-i\sqrt{2\pi}q^\alpha \exp\left(\frac{2(\log (-iq^\alpha))^2}{\log q^{-1}}\right)}
{q^\frac18(1-q)\sqrt{\log q^{-1}}(q;q)_\infty^3}
\frac{\Gamma_q(\frac12,\frac32)}{\Gamma_q(\frac54+a,\frac54+b,\frac34+a,\frac34+b)}
%\nonumber\\
%&&\hspace{6.7cm}\times
\qpsihyp46{\pm q^\frac54,q^{\frac14-a},q^{\frac14-b}}{\pm q^\frac14,q^{\frac54+a},q^{\frac54+b},0,0}{q,q^{\frac32+a+b}}.\nonumber\\
\label{qbetaCgam3}
\end{eqnarray}
Now if one takes the limit as $q\to 1^{-}$ of \eqref{qbetaCgam3} using
\eqref{limit}, one can see that the limit vanishes because 
\[
\lim_{q\to1^{-}}\frac{(1-q)^{a+b-1}}{\Gamma_q(a+b+1)}=0.
\]
Another way one can see that the limit vanishes is noting that the ${}_4\psi_6$ reduces to a ${}_3H_3(1)$ which can be evaluated using \eqref{5:3H3} and therefore the argument of one of the denominator gamma functions necessarily vanishes. Hence the whole ${}_3H_3(1)$ necessarily vanishes. Using a similar argument one can see that the $q\to1^{-}$ limits of the integrals representations of the ${}_3\psi_6$ 
\begin{eqnarray}
&&\hspace{-1.3cm}\int_{-\infty}^\infty(1+q^{2x}\alpha^2)(-q^{x+1}\alpha a,\frac{q^{1-x}a}{\alpha};q)_\infty\,q^{2x^2-x}\alpha^{4x}\,\dd x=\frac{\sqrt{2\pi}\alpha\exp\left(\frac{2(\log \alpha)^2}{\log q^{-1}}\right)}{q^\frac18\sqrt{\log q^{-1}}} \nonumber\\
&&\hspace{1.0cm}=\frac{\sqrt{2\pi}\alpha\exp\left(\frac{2(\log \alpha)^2
}{\log q^{-1}}\right)(iq^\frac54a,iq^\frac34a;q)_\infty}{q^\frac18\sqrt{\log q^{-1}}(q,q^\frac12,q^\frac32;q)_\infty}
\qpsihyp36{\pm q^\frac54,-\frac{iq^\frac14}{a}}{\pm q^\frac14,iq^\frac54a,0,0,0}{q,iq^\frac74a},
\end{eqnarray}
and the ${}_2\psi_6$ 
\begin{eqnarray}
&&\hspace{-4.3cm}\int_{-\infty}^\infty(1+q^{2x}\alpha^2)\,q^{2x^2-x}\alpha^{4x}\,\dd x=\frac{\sqrt{2\pi}\alpha\exp\left(\frac{2(\log \alpha)^2}{\log q^{-1}}\right)}{q^\frac18\sqrt{\log q^{-1}}} \nonumber\\
&&\hspace{-2.0cm}=\frac{\sqrt{2\pi}\alpha\exp\left(\frac{2(\log \alpha)^2
}{\log q^{-1}}\right)}{q^\frac18\sqrt{\log q^{-1}}(q,q^\frac12,q^\frac32;q)_\infty}
\qpsihyp26{\pm q^\frac54}{\pm q^\frac14,0,0,0,0}{q,q^2},
\end{eqnarray}
must also vanish.

\medskip
\noindent One interesting open question is 
whether there exists more general integral representations of basic bilateral series of Ramanujan-type, such as exists in the $q\to1^{-}$ limit. We are as of yet unable to find such generalizations.

%%%%%%%%%%%%%%%%%%%%%%%%%%%%%%%%%%%%%%%%%%%%%%%%%%%%%%%%%%%%%
%\subsection*{Acknowledgements}
%We would like to thank Mourad Ismail and Keru Zhou for valuable discussions. \\[-0.8cm]
%%%%%%%%%%%%%%%%%%%%%%%%%%%%%%%%%%%%%%%%%%%%%%%%%%%%%%%%%%%%%
%%%\section*{References}
%\bibliographystyle{plain}
%\bibliography{refbib} 

\begin{thebibliography}{10}

\bibitem{AAR}
G.~E. Andrews, R.~Askey, and R.~Roy.
\newblock {\em Special functions}, volume~71 of {\em Encyclopedia of
  Mathematics and its Applications}.
\newblock Cambridge University Press, Cambridge, 1999.

\bibitem{CohlVolkmer2024}
H.~S. {Cohl} and H.~{Volkmer}.
\newblock {Bilateral discrete and continuous orthogonality relations in the
  $q^{-1}$-symmetric Askey scheme}.
\newblock {\em {\tt Submitted}}, 2024.

\bibitem{Conway1978}
J.~B. Conway.
\newblock {\em Functions of one complex variable}, volume~11 of {\em Graduate
  Texts in Mathematics}.
\newblock Springer-Verlag, New York-Berlin, second edition, 1978.

\bibitem{NIST:DLMF}
{\it NIST Digital Library of Mathematical Functions}.
\newblock \href{https://dlmf.nist.gov/}{{\bf\tt\normalsize
  https://dlmf.nist.gov/}}, Release 1.2.2 of 2024-09-15.
\newblock F.~W.~J. Olver, A.~B. {Olde Daalhuis}, D.~W. Lozier, B.~I. Schneider,
  R.~F. Boisvert, C.~W. Clark, B.~R. Miller, B.~V. Saunders, H.~S. Cohl, and
  M.~A. McClain, eds.

\bibitem{GaspRah}
G.~Gasper and M.~Rahman.
\newblock {\em Basic hypergeometric series}, volume~96 of {\em Encyclopedia of
  Mathematics and its Applications}.
\newblock Cambridge University Press, Cambridge, second edition, 2004.
\newblock With a foreword by Richard Askey.

\bibitem{Grafakos2014}
L.~Grafakos.
\newblock {\em Classical {F}ourier analysis}, volume 249 of {\em Graduate Texts
  in Mathematics}.
\newblock Springer, New York, third edition, 2014.

\bibitem{IsmailRahman1995}
M.~E.~H. Ismail and M.~Rahman.
\newblock Some basic bilateral sums and integrals.
\newblock {\em Pacific Journal of Mathematics}, 170(2):497--515, 1995.

\bibitem{Knopp1990}
K.~Knopp.
\newblock {\em Theory and Application of Infinite Series}.
\newblock Dover Books on Mathematics. Dover Publications, 1990.

\bibitem{Nielsen1906}
N.~Nielsen.
\newblock {\em Die {G}ammafunktion. {B}and {I}. {H}andbuch der {T}heorie der
  {G}ammafunktion. {B}and {II}. {T}heorie des {I}ntegrallogarithmus und
  verwandter {T}ranszendenten}.
\newblock Chelsea Publishing Co., New York, 1965.

\bibitem{Ramanujan1920}
S.~Ramanujan.
\newblock A class of definite integrals [{Q}uart. {J}. {M}ath. {\bf 48} (1920),
  294--310].
\newblock In {\em Collected papers of {S}rinivasa {R}amanujan}, pages 216--229.
  AMS Chelsea Publ., Providence, RI, 2000.

\bibitem{Slater66}
L.~J. Slater.
\newblock {\em Generalized hypergeometric functions}.
\newblock Cambridge University Press, Cambridge, 1966.

\end{thebibliography}

\def\cprime{$'$} \def\dbar{\leavevmode\hbox to 0pt{\hskip.2ex \accent"16\hss}d}

\end{document}